\documentclass[a4paper, 11pt]{article}
\usepackage{dsfont}

\usepackage{calrsfs}
\usepackage{comment}
\usepackage{ulem}

\usepackage{amsmath, amsthm, amssymb, amsfonts}

\usepackage[shortlabels]{enumitem}
\usepackage[square]{natbib}
\usepackage{float}

\usepackage[usenames]{color}

\usepackage{tikz}

\usetikzlibrary{arrows,decorations.pathmorphing,backgrounds,positioning,fit,automata}
\usetikzlibrary{petri}
\usetikzlibrary{topaths}

\usepackage[colorlinks=true,breaklinks=true,bookmarks=true,urlcolor=blue,
     citecolor=blue,linkcolor=blue,bookmarksopen=false,draft=false]{hyperref}
     
\newcommand{\ignore}[1]{}

\newtheorem{theorem}{Theorem}

\newtheorem{lemma}[theorem]{Lemma}

\newtheorem{proposition}[theorem]{Proposition}

\newtheorem{definition}[theorem]{Definition}

\newtheorem{remark}[theorem]{Remark}

\DeclareMathOperator{\card}{Card}
\begin{document}

\title{General distribution of consumers in pure Hotelling games}

\author{
Ga\"etan Fournier\thanks{This research was supported by the ISF grant 1585/15}\\
School of Mathematical Sciences,
Tel-Aviv University\\
Schreiber Building, Tel Aviv, Israel, 6997800\\
\texttt{fournier.gtn@gmail.com}
}

\maketitle

\begin{abstract}
A pure Hotelling game is a competition between a finite number of players who select simultaneously a location in order to attract as many consumers as possible. In this paper, we study the case of a general distribution of consumers on a network generated by a metric graph. Because players do not compete on price, the continuum of consumers shop at the closest player's location. Under regularity hypothesis on the distribution we prove the existence of an $\epsilon$-equilibrium in pure strategies and we construct it, provided that the number of players is larger than a lower bound.
\end{abstract}



\section{Introduction}

\subsection{General and pure Hotelling games}

The seminal paper of \cite{hotelling} introduced a model of spatial competition that we now refer to as Hotelling games. It considers the competition between two retailers along a segment where consumers are assumed to be uniformly distributed. In the unique equilibrium both players set their shop in the middle of the segment. A large literature generalized the model and two different classes appeared: the general Hotelling game, where players can decide of their location and also of a selling price. In this framework, a cost function is introduced to model the cost of moving for the consumers. The existing results mostly concern very simple sets of possible locations, as described in the following review of the literature. On the opposite, another part of the literature studies further the case of pure Hotelling game, where the price is not under control of the players. This model applies particularly to the sell of products whose price is fixed, such as newspapers sellers or brand products retailers.\\

The current paper belongs to the second literature, where the assumption of fixed prices makes possible the analysis of the interaction between a larger number of players on more evolved networks for players and consumers. Even though the study of the segment is useful to understand better the dynamics of the interactions, it appears that the network plays an important role in the problem. The network we study in this paper is the one introduced in \cite{palvolgyi11} and \cite{fournier14} and can model a city, where consumers are distributed along different streets that eventually intersect each other. We use a metric graph and suppose that consumers are distributed along the edges, i.e. at convex combinations between two linked vertices. Players can choose a location anywhere in this network and consumers buy a fixed quantity of goods at the closest store. The payoff of the players is the amount of consumers that shop at their store.\\

Pure Hotelling games often assume that consumers are uniformly distributed along the network. This hypothesis simplifies the analysis but is restrictive: such a model does not take into account the  irregularities inherent to a city, such as higher density in city centers. This consideration can not be handled by simple modifications on the network (using dilatation or contraction arguments for example): players care about the distance with respect to the distribution of consumers but consumers are going to the closest player's location regardless of the distribution of the other consumers.\\ 

In this paper, we consider a general distribution of consumers and only assume that this distribution is absolutely continuous with respect to the Lebesgue measure, that can easily be defined on the metric graph. We prove though counterexamples that no general results hold on the existence of exact Nash equilibrium, neither on the existence of approximate Nash equilibrium for an arbitrary number of players. Our main result is that if the distribution of consumers has a positive, bounded and smooth density with respect to the Lebesgue measure, then for any $\epsilon>0$ there exists an $\epsilon$-equilibrium in pure strategies in the pure Hotelling game, provided that the number of players is larger than a lower bound. We give a constructive proof and an explicit bound $N(\epsilon)$ on the number of players. The $\epsilon$-equilibrium is both additive and multiplicative.\\

A brief survey of the literature on pure Hotelling games is made in subsection \ref{subse:literature}. We give a formal description of the problem in \ref{se:model}. We state the main result in section \ref{se:result}. In section \ref{se:examples} we provide the announced counterexamples. The section \ref{se:proof} is devoted to the proofs.\\

\subsection{Literature}\label{subse:literature}

We now discuss the closely related articles, that have in common with the current paper either to be concerned with the class of pure Hotelling games (where only the choice of location is under the control of the players), or to study the case of non-uniform distribution of consumers, or to consider interactions on networks similar to the one considered in the current paper.\\

Let first discuss the case of papers concerned with pure Hotelling games. \cite{marc} consider a uniform distribution of consumers on the unit interval. Consumers take into account their travel distances to their closest player's location, but also their expected queuing times, which depend on the number of consumers choosing the same firm. They provide existence of a refined equilibrium in several particular cases where the number of players is small and even.
\cite{nunez2014competing} and \cite{nunez2015large} consider a pure Hotelling game with a finite set of predefined possible locations for the sellers.
In a mimeo, \cite{palvolgyi11} considers pure Hotelling games on the same network that the current paper, but with uniform distribution of consumers. He states that there always exists a pure Nash equilibrium when the number of players is large enough and provides a constructive proof in some particular cases. He also studies some equilibrium structural properties.
\cite{fournier14} give a general proof for the existence of a pure Nash equilibrium with a large number of players and a uniform distribution of consumers. They also measure the efficiency of the equilibrium. Similarities and differences between \cite{palvolgyi11}, \cite{fournier14} and the current paper methods for constructing an equilibrium with the adequate properties will be discussed in subsection \ref{subsec:eq_step}.\\

We now discuss papers concerned with non uniform distribution of consumers in Hotelling games. \cite{eaton75} study, among a lot of different models, the case of a general distribution of consumers on the unit interval, in the slightly different setting where two players cannot choose the same location. They show that Hotelling's original results only hold under strong hypotheses.
\cite{lederer1986competition} consider a model where consumers are non uniformly distributed on the plane, and study the interactions between two different firms. 
\cite{neven86} analyzes the general location-then-price Hotelling game with two players on the unit interval. He considers increasing densities of consumers towards the center and shows that for some distribution, the two firms locate at opposite ends of the market.
\cite{anderson97} also examines a duopoly in a general Hotelling game, with log-concave density of consumers. He proves that there exists a unique equilibrium in pure strategies if the density is not "too asymmetric" and not "too concave".
\cite{montes15} considers a two stages location-then-price game on the unit interval with non-uniform distribution of consumers. Under strong conditions on the distribution, he studies a refinement of Nash equilibrium and provides an algorithm to analyze some comparative statics.\\

We now discuss the papers concerned with Hotelling games on a network similar to the one considered in the current paper. \cite{palvolgyi11} and \cite{fournier14} are already mentioned above. \cite{heijnen2014price} consider such a network and study the general location-then-price Hotelling game with 2 players.
Some papers consider the graph structure, but suppose that players can locate only on the vertices of the graph, such as \cite{mavronicolas2008voronoi} or \cite{feldmann2009nash}.

\section{The game $\mathcal{H}(n,G,f)$}\label{se:model}

In this section we give a formal description of the model. $\mathcal{H}(n,G,f)$ is the game where $n \geq 2$ players\footnote{In the case where $n=1$ player, every location attracts all the consumers and is therefore a pure equilibrium.} select simultaneously a location in a set $G$, and where the consumers are distributed on $G$ according to the density function $f$ with respect to the Lebesgue measure. The network $G$ is defined in subsection \ref{subse:network}, and the normal form of the game, including payoffs, is given in subsection \ref{subse:normal_form}.

\subsection{The network $G$}\label{subse:network}

The network $G$ on which consumers are distributed is also the set of possible actions for the players. To define formally $G$, we need a triplet  $(V,E,\lambda)$ where $V$ is a finite set of vertices, $E$ is a finite set of edges and $\lambda$ is a vector of length on $E$. For simplicity we suppose the graph $(V,E)$ to be connected. If the edge $e \in E$ links the vertices $u$ and $v$ we denote it $(u,v)$. In our network, players and consumers are not only located in the vertices of the graph but also along the edges. A point along the edge $(u,v)$ is defined as a convex combination of $u$ and $v$. The network $G$ is then defined as the set  $G:=\{(u,v,t) | u,v \in V, (u,v) \in E, t \in [0,1]\}$, and a point in $G$ is called a possible location. If a possible location $y$ is chosen by at least one player for a given actions' profile, it is then referred to as a location (or chosen location) of this profile.  We give an arbitrary orientation to every edge to avoid the confusion between $(u,v,t)$ and $(v,u,1-t)$ that both refer to the same location.\\

We now endow the set $G$ with a metric $d$ that represents the distance that consumers have to travel to go from one point to another in the network. For a given vector of length $\lambda=(\lambda_e)_{e \in E}$, $\lambda_e>0$ is referred to as the length of the edge $e$. From $\lambda$, we derive a distance $d$ on $G$ as follows: if $x_1$ and $x_2$ are 2 points in the same edge $e$, they can be written as $x_{1}=(u,v,t_{1})$ and $x_{2}=(u,v,t_{2})$ for some $t_1,t_2 \in [0,1]$. The distance between $x_1$ and $x_2$ is $d(x_1,x_2)= \lambda_{(u,v)} \times  \vert t_2 - t_1 \vert$. If $x$ and $y$ are not in the same edge, we consider $P(x,y)$ the set of paths between $x$ and $y$ as the set of all sequences $(x_1,\dots,x_n)$ with finite $n$ such that $x_1=x$, $x_n=y$ and such that for all $i \in \{1,\dots,n-1\}$, $x_i$ and $x_{i+1}$ belong to the same edge. The distance $d(x,y)$ is then defined as:

$$d(x,y):= \inf_{(x_1,\dots,x_n)\in P(x,y)} \sum_{i=1}^{n-1} d(x_i,x_{i+1})$$

Given a profile of actions $(x_1,\dots,x_n) \in G^n$ for the players, we can compute how the consumers split: they shop to the closest player's location (ties are discussed in subsection \ref{subse:normal_form}). \\

The distribution of consumers plays its role in the players' payoffs. We first introduce the useful notion of interval $[x_1,x_2] \subset (u,v)$: let $x_1:=(u,v,t_1)$, $x_2:=(u,v,t_2)$ and suppose without loss of generality that $t_1<t_2$. Then $[x_1,x_2]$ is defined as the set of $\{(u,v,t)\}$ for $t\in [t_1,t_2]$.\footnote{Respectively, we define $[x_1,x_2[$, $]x_1,x_2]$, and $]x_1,x_2[$ as the set of $\{(u,v,t)\}$ for $t\in [t_1,t_2[$, $]t_1,t_2]$, or $]t_1,t_2[$.} The terminology interval comes from the straightforward isometric identification between $[x_1,x_2]$ and the real interval $[t_1 \lambda_e,t_2 \lambda_e]$ when the edge $e$ is fixed. If there is no possible confusion on $e$, we sometimes use the abuse of notation $[t_1 \lambda_e,t_2 \lambda_e]$ to denote the interval $[x_1,x_2]$.\\

The Lebesgue measure on $G$ is denoted $\mathcal{L}$ and can easily be defined as a natural extention of the Lebesgue measure on a real interval. A subset of $G$ can indeed be identified with a finite union of subsets of intervals. The uniform distribution is the case where $G$ is endowed with the measure $\mathcal{L}$. In our model, it corresponds to the case where $f=\textbf{1}$, i.e. $f(x)=1$ for every $x \in G$. This situation is studied in \cite{palvolgyi11} and \cite{fournier14}. In this paper, we extend the model to a general non-atomic distribution $\mu$ with density function $f: G \rightarrow \mathbb{R}^{+}$ with respect to $\mathcal{L}$. For a subset $\mathcal{A} \subset G$, the quantity of consumers located in $\mathcal{A}$ is:

$$\int_{\mathcal{A}} f(x) d\mathcal{L}(x)$$

For simplicity, we sometimes use the arbitrary orientation to identify the restriction of $f$ on an edge $e$ with a real function $f_e: [0,\lambda_e] \rightarrow \mathbb{R}^{+}$. When 
three chosen locations $x_1,x_2,x_3$ belong to the same edge $e$, and if the orientation is such that $0 \leq x_1<x_2<x_3 \leq \lambda_e$, we say that the location $x_1$ (resp. $x_3$) is the left (resp. right) neighbor of $x_2$ if there is no other chosen location between $x_1$ and $x_2$ (resp. between $x_2$ and $x_3$).\\

Finally remark that we can remove any vertex $v$ with degree 2 and consider the new edge $(u,w)$, instead of $(u,v)$ and $(v,w)$. In this case we set $\lambda_{(u,w)}=\lambda_{(u,v)}+\lambda_{(v,w)}$. From now on, we can assume without loss of generality that all vertices have a degree different from 2. 

\subsection{The normal form of the game $\mathcal{H}(n,G,f)$}\label{subse:normal_form}

$\mathcal{H}(n, G, f)$ is the Hotelling game on $G$ with $n$ players and with a distribution $f$ of consumers. Its normal form is given by a finite set of players $\{1, \dots, n\}$, the action set $G$ which is the same for every player, and the payoff function $p:=(p_i)_{i \in \{1,\dots,n\}}$ that depends on the chosen profile of actions. We first give an intuitive definition of the payoffs, and equation $(\ref{eq:payoff_function})$ gives a formal definition.\\

The payoff of a player is the quantity of consumers that are going to his shop. A location $x_i \in G$ attracts all the consumers who are closer to $x_i$ than to any other location, plus a share of the consumers who are as close to $x_i$ as to some other locations (all these locations get an equal part of the consumers). The total quantity of consumers that shop in $x_i$ splits equally between the different players in $x_i$.\\

The payoff function has to consider that a location can have different ties with different other locations. More precisely, for a subset $A$ of an actions' profile $\{x_1,\dots,x_n\}$, $D_{A}$ is the set of point in $G$ that are at the same distance from 
all locations $x_i \in A$ and that are strictly closer to them than to any other location. Formally, for $A \subset \{ x_1,\dots,x_n \}$:
\begin{equation}\label{def:C_A}
\begin{split}
D_{A} &= \{y\in G: d(y,x_{i}) = d(y,x_{j}) \text{ for all } x_{i},x_{j}\in A \\
&\qquad\text{ and } d(y,x_{i}) < d(y,x_{\ell})  \text{ for all }x_{i}\in A, x_{\ell} \not\in A\}.
\end{split}
\end{equation}

The domain of attraction of a location $x_i \in G$ refers to the points in $G$ that are (weakly) closer to them than to other locations, i.e. $\displaystyle \bigcup_{{A \subset \{x_1,\dots,x_n \} \atop x_{i}\in A }} D_{A}$.\\

All locations in $A$ get an equal part of $\mu(D_A)$. Moreover, all players in a same location get the same payoff. Formally, the payoff of the player $i$ when the profile of strategy $\boldsymbol{x}=(x_1,\dots,x_n)$ is played is:
 
\begin{equation}\label{eq:payoff_function}
p_{i}(\boldsymbol{x}) = \frac{1}{\card(\{j\in \{1,\dots,n \}: x_{j}=x_{i}\})}\sum_{A \subset \{x_1,\dots,x_n \} \atop x_{i}\in A } \frac{\mu(D_A)}{\card(A)}.
\end{equation}
where $\mu(D_A)= \int_{D_A} f(x) d\mathcal{L}(x)$. From now on, we use the notation $dx$ instead of $d\mathcal{L}(x)$.\\ 

This payoff function has a symmetry property: the payoff of a player does not depend on the identities of the other players. Therefore, the fact that a profile $\boldsymbol{x}$ is or not an equilibrium does not depend on the identities of the players. We sometimes use the term configuration to talk about a profile of actions.\\

\section{Existence of $\epsilon$-equilibrium in $\mathcal{H}(n,G,f)$}\label{se:result}

In this section, we state our main result on the existence of $\epsilon$-equilibrium in $\mathcal{H}(n, G, f)$ in pure strategies.\\

\begin{definition}
A profile of actions $\boldsymbol{x}:=(x_1,\dots,x_n)$ is a multiplicative $\epsilon$-equilibrium of $\mathcal{H}(n,G,f)$ in pure strategies if and only if for all $i \in \{1,\dots,n\}$ and for all $y \in G$ we have:

$$ p_{i}(x_1,\dots,x_{i-1},y,x_{i+1},\dots,x_n) \leq  (1+\epsilon) p_{i}(\boldsymbol{x}) $$

$\boldsymbol{x}$ is an additive $\epsilon$-equilibrium of $\mathcal{H}(n,G,f)$ in pure strategies if and only if for all $i \in \{1,\dots,n\}$ and all $y \in G$, we have:

$$ p_{i}(x_1,\dots,x_{i-1},y,x_{i+1},\dots,x_n) - p_{i}(\boldsymbol{x}) \leq  \epsilon $$
\end{definition}

~~\\

We now define two conditions on the density function $f$.\\

\noindent \textbf{(C1)}: The function $f$ is $K$-Lipschitz with respect to the distance $d$ on $S$.\\

\noindent \textbf{(C2)}: The function $f$ is positively upper and lower bounded: there exist $m,M \in ]0,\infty[$ such that for all $y \in G$, $0 < m \leq f(y) \leq M$

\begin{theorem}{Existence of $\epsilon$-equilibrium.\\}\label{thm:main_thm} 
Let $G$ be the network generated by any connected graph $(V,E)$ and $\lambda$ be any vector of lengths. Suppose that $f$ satisfies \textbf{(C1)} and \textbf{(C2)}. For all $\epsilon>0$ there exists an integer $N(\epsilon)$ such that when the number of players $n$ is larger than $ N(\epsilon)$, there exists a multiplicative $\epsilon$-equilibrium in pure strategies in $H(n,G,f)$. \\

\noindent We give a constructive proof where $N(\epsilon)$ can be taken equal to:\footnote{In particular that $N(\epsilon) \sim \frac{1}{\epsilon}$ when $\epsilon \rightarrow 0$} $$  5 \card(E) +  \left\lceil \frac{5L (M+\frac{\epsilon m}{12 + \epsilon})}{m-\frac{\epsilon m}{12 + \epsilon}}\left(\frac{(12 + \epsilon)K}{2 \epsilon m} + \frac{1}{\min \lambda_e}\right) +   \frac{3 L K (12 + \epsilon)}{2 \epsilon m} \right\rceil$$ where $L:= \int_S f(x)dx$ is the total quantity of consumers in the network, and $\min \lambda_e$ denotes the minimal length among the set $E$ of edges.
\end{theorem}

\begin{remark}
We can weaken condition \textbf{(C1)} to the following condition: The restriction of $f$ on $(0,\lambda_e)$, the interior of any edge $e$, is $K$-Lipschitz. With such a condition, the function $f$ can be discontinuous in the vertices.
\end{remark}

\begin{remark}
There exists also a lower bound $N'(\epsilon)$ on the number of players that guarantees the existence of an additive $\epsilon$-equilibrium in pure strategies in $H(n,G,f)$. Because the payoff of any player is always bounded by $L$, this existence is just a corollary of Theorem \ref{thm:main_thm}. Nevertheless, such a result is somehow trivial: when the number of players is large, their payoffs are small because they share the fixed quantity $L$, and any profile of actions turns to be an additive equilibrium.
\end{remark}

\begin{proof}

The proof of theorem \ref{thm:main_thm} is detailed in section \ref{se:proof}.  We give here a sketch of the proof.\\

The first step is detailed in subsection \ref{subsec:approx}: we approximate the density function $f$ by a step function $g(\epsilon_1)$, where $\epsilon_1$ is a parameter playing a role in the length of the steps. We prove that because $f$ is $K$-Lipschitz, the step function $g(\epsilon_1)$ is such that $\left\| f-g(\epsilon_1) \right\|_{\infty} \leq \epsilon$ when $\epsilon_1$ is small enough.\\

The second step is detailed in subsection \ref{subsec:eq_step}: we give a constructive proof that if \textbf{(C2)} holds, there exists an (exact) equilibrium in pure strategies in the game $\mathcal{H}(n,S,g(\epsilon_1))$, where consumers are distributed according to the density function $g(\epsilon_1)$, and where the number of players $n$ is larger than a lower-bound $N(\epsilon_1)$. This lower bound increases when $\epsilon_1$ goes to zero.\\

We conclude in subsection \ref{subsec:eps_eq} by showing that if $\epsilon_1$ is small enough, the equilibrium constructed in the previous step is a multiplicative $\epsilon$-equilibrium in the game $\mathcal{H}(n,S,f)$, where the distribution of the consumers is given by the function $f$. We obtain therefore a lower bound $N(\epsilon)$ on the number of players $n$ that guarantees the existence of a pure multiplicative $\epsilon$-equilibrium in $\mathcal{H}(n,S,f)$.
\end{proof}

\section{A few counterexamples}\label{se:examples}

In this section we provide counterexamples to naive extensions of the state-of-the-art results. The counterexample also illustrate how the analysis of game $\mathcal{H}(n,S,f)$ with a general distribution of consumers $f$ is fundamentally different from the analysis of the game with uniform distribution of consumers $\mathcal{H}(n,S,\textbf{1})$, where $\textbf{1}$ stands for the constant function $x \in G \mapsto 1$. They also highlight the necessity to consider games with a large enough number of players.\\

Proposition \ref{ex:no_eq_monotonic} shows that on the simple network composed of only one edge of length $1$, identified with $[0,1]$, it is possible to have existence of exact equilibrium in $\mathcal{H}(n,[0,1],\textbf{1})$ but non existence in $\mathcal{H}(n,[0,1],f)$ even when $f$ is arbitrary close to $\textbf{1}$. More precisely, it is proved in \cite{eaton75} that for any $n \geq 4$ there exists an equilibrium in pure strategies in the game $\mathcal{H}(n,[0,1],\textbf{1})$. In the case of non uniform distribution of consumers, it is claimed that in the slightly different framework where two players can not chose the same location, there exist some density distributions $f$ such that the game $\mathcal{H}(n,[0,1],f)$ doesn't admit any equilibrium in pure strategies for $n \geq 3$. We give here a more precise result and provide an extensive proof that also covers the case where several players are allowed to play in the same location. Proposition \ref{ex:no_eq_monotonic} is also a counterexample to a possible generalization of the existence of exact equilibrium for a large number of player, as proved in \cite{fournier14} for uniform distribution.

\begin{proposition}\label{ex:no_eq_monotonic} For any $\epsilon>0$, there exists a density function $f$ such that $\| f- \textbf{1}\|_{\infty} \leq \epsilon$ and such that $\mathcal{H}(n,[0,1],f)$ doesn't admit any Nash equilibrium in pure strategies for $n \geq 3$. \end{proposition}

\begin{proof}
Fix an $\epsilon>0$ and let $f: [0,1] \rightarrow \mathbb{R}^+$ defined by $f(x):= 1+\epsilon x$. It is clear that $\| f- \textbf{1}\|_{\infty} \leq \epsilon$. Suppose now that $\boldsymbol{x}=(x_1,\dots,x_n)$ is an equilibrium. Without loss of generality we assume that $x_1 \leq \dots \leq x_n$. 

We first claim that all players are coupled, i.e. that $x_1=x_2<x_3=x_4<\dots<x_{n-1}=x_n$.\footnote{This results implies in particular that there is no equilibrium with an odd number of players.}

Suppose that $m \geq 3$ players share the same location $x_k=x_{k+1}=\dots=x_{k+m-1}$ with $m \geq 3$ players. $x_{k-1}$ (resp. $x_{k+m}$) is its left (resp. right) neighbor (if it has one). We denote $x_{k-1,k}=\frac{x_{k-1}+x_k}{2}$ and $x_{k+m-1,k+m}=\frac{x_{k+m-1}+x_{k+m}}{2}$. In the case where $x_k$ does not have a left (resp. right) neighbor we set $x_{k-1,k}=0$  (resp $x_{k+m-1,k+m}=0$).

$D_{x_k}$, the domain of attraction of $x_k$, is equal to the union of the two intervals $[x_{k-1,k},x_k]$ and $[x_{k+m-1},x_{k+m-1,k+m}]$ (see the red and green intervals in figure \ref{fig:cex} below).

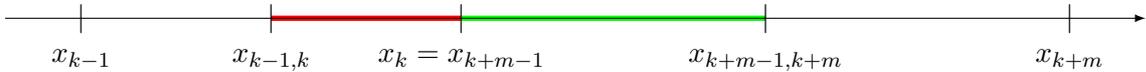
\begin{figure}[H]
\centering

\begin{tikzpicture}[>=latex]

 \draw[draw=white,double=red,double distance=2pt,line width=2pt] (3.5,0)-- (5.99,0);
    
     \draw[draw=white,double=green,double distance=2pt,line width=2pt] (6,0)-- (10,0);
     
    \draw[->] (0,0) --(15,0);
    \node[] at (1,0) {$|$};
    \node[] at (3.5,0) {$|$};
    \node[] at (6,0) {$|$};
    \node[] at (10,0) {$|$};
    \node[] at (14,0) {$|$};
    
    \node[below=8pt] at (1,0) {$x_{k-1}$};
    \node[below=8pt] at (3.5,0) {$x_{k-1,k}$};
    \node[below=8pt] at (6,0) {$x_k=x_{k+m-1}$};
    \node[below=8pt] at (10,0) {$x_{k+m-1,k+m}$};
    \node[below=8pt] at (14,0) {$x_{k+m}$};
 %   \node[below=8pt] at (11,0) {$2$};

  %  \foreach \xp in {3.6,3.7,...,5.9}{\node[] at (\xp,0) {\begin{color}{purple} / \end{color}};}
 %   \foreach \xp in {6.1,6.2,...,9.9}{\node[] at (\xp,0) {\begin{color}{purple} $\backslash$ \end{color}};}

\end{tikzpicture}
%~\vspace{0cm} \caption{\label{fi:1} Players in the interior step $[i \ell_e(\epsilon_1),(i+1) \ell_e(\epsilon_1)]$.}

\caption{\label{fig:cex} Domain of attraction of $x_k$}
\end{figure}

The payoff of player $k$ is then equal to:

$$p_k(x_1,\dots,x_n)=\frac{1}{m} \int_{[x_{k-1,k},x_k] \cup [x_k,x_{k+m-1,k+m}]} (1+\epsilon x) dx$$

Because $m \geq 3$, we have either: $$\int_{[x_{k-1,k},x_k]} (1+\epsilon x) dx > p_k(x_1,\dots,x_n)$$ or: $$\int_{[x_{k},x_{k+m-1,k+m}]} (1+\epsilon x) dx > p_k(x_1,\dots,x_n)$$ 
But player $k$ could get a payoff arbitrary close to $\int_{[x_{k-1,k},x_k]} (1+\epsilon x) dx$ by playing $x_k-\delta$ for a $\delta$ small enough, or arbitrary close to $\int_{[x_{k},x_{k+m-1,k+m}]} (1+\epsilon x) dx$ by playing $x_{k}+\delta$ for a $\delta$ small enough. It proves that at equilibrium, it is not possible that $3$ or more players share a location.\\

Suppose now that there exists a location $x_k \in [0,1]$ with a single player $k$. This player has a left neighbor $x_{k-1}$ (resp. a right neighbor $x_{k+1}$), otherwise he would have a profitable deviation playing $x_k + \delta$ (resp. $x_k - \delta$) for a small enough $\delta$. His payoff is equal to a right trapezoid's area (see figure \ref{fig:cex2} below):

$$p_k(x_1,\dots,x_n)= (x_{k,k+1}-x_{k-1,k} )\left(1+\epsilon \frac{x_{k,k+1}+x_{k-1,k}}{2}\right)$$

\begin{figure}[H]
\centering

\begin{tikzpicture}[>=latex]

\draw[draw=white,double=red,double distance=2pt,line width=2pt] (6,0)-- (8,0);
    
\draw[draw=white,double=green,double distance=2pt,line width=2pt] (8,0)-- (10,0);
     
\draw[->] (0,0) --(15,0);
\draw[->] (0,0) --(0,2.5);

\draw[dash pattern=on 1mm off 1mm] (0,1) -- (6,1);
\draw[dash pattern=on 1mm off 1mm] (0,1.5) -- (10,1.5);

    \node[] at (4,0) {$|$};
    \node[] at (6,0) {$|$};
    \node[] at (8,0) {$|$};
    \node[] at (10,0) {$|$};
    \node[] at (12,0) {$|$};
    
    \draw[-] (6,1) --(6,0);
    \draw[-] (6,1) --(10,1.5);
    \draw[-] (10,0) --(10,1.5);
    
    \node[left=8pt] at (0,1) {$1+\epsilon  x_{k-1,k}$};
    \node[left=8pt] at (0,1.5) {$1+\epsilon  x_{k,k+1}$}; 
 %   \node[left=2pt] at (0,2.5) {$f(x)$};    

    \node[below=8pt] at (4,0) {$x_{k-1}$};
    \node[below=8pt] at (6,0) {$x_{k-1,k}$};
    \node[below=8pt] at (8,0) {$x_{k}$};
    \node[below=8pt] at (10,0) {$x_{k,k+1}$};
    \node[below=8pt] at (12,0) {$x_{k+1}$};

\end{tikzpicture}
\caption{\label{fig:cex2} Domain of attraction and payoff of player $k$ in $x_k$}
\end{figure}
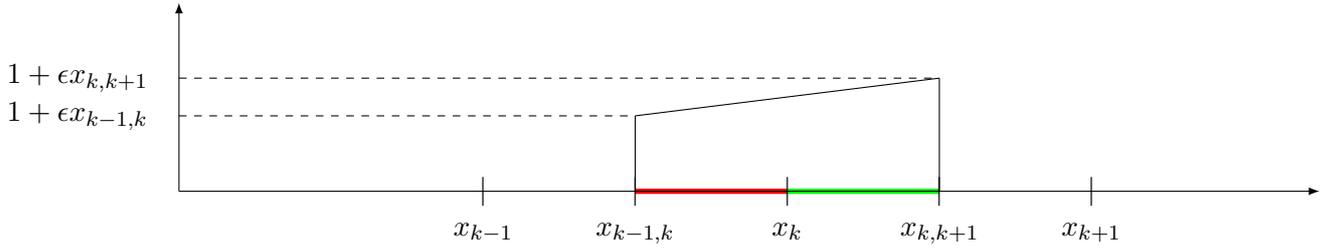

For a small enough $\delta>0$, we have that:

$$p_k(x_1,\dots,x_{k-1},x_k+\delta,x_{k+1},\dots,x_n)= (x_{k,k+1}-x_{k-1,k} )\left(1+\epsilon \frac{x_{k,k+1}+x_{k-1,k}}{2} + \epsilon \delta \right) > p_k(x_1,\dots,x_n)$$

The last inequality is in contradiction with $(x_1,\dots,x_n)$ being an equilibrium. We therefore proved that if $(x_1,\dots,x_n)$ is an equilibrium, then all players are coupled.\\

We now suppose that all players are coupled $x_1=x_2$, $x_3=x_4$, $\dots,$, $x_{n-1}=x_n$, and will find a contradiction. We denote $A_1:= \int_0^{x_1} f(x) dx$, $A_2:=\int_{x_1}^{x_{2,3}} f(x) dx$,$\dots$, $A_{2k-1}=\int_{x_{2k-2,2k-1}}^{x_{2k-1}} f(x) dx$, $A_{2k}=\int_{x_{2k}}^{x_{2k,2k+1}} f(x) dx$, $\dots$, $A_n:= \int_{x_n}^{1} f(x) dx$. We now prove $A_2=A_3$ and that it implies a contradiction.\\

We have $p_1(\boldsymbol{x})=p_2(\boldsymbol{x})=\frac{A_1+A_2}{2}$. If $A_1>A_2$ (resp. $A_2>A_1$) then player 1 could deviate to $x_1-\delta$ (resp. $x_1+\delta$) and get a payoff arbitrary close to $A_1$ (resp. $A_2$) $\geq \frac{A_1+A_2}{2}$. The equilibrium condition gives $A_1=A_2=p_1(\boldsymbol{x})$.  But $A_1<A_3$ is forbidden, otherwise player 1 could deviate to $x_3-\delta$ and have a payoff arbitrary close to $A_3$. A similar argument forbids $A_3<A_1$. We therefore have $A_1=A_2=A_3$. But $A_2=A_3$ is not possible because these two quantities are defined as integral of a strictly increasing function on two intervals of the same length. Therefore there is no equilibrium in the game $\mathcal{H}(n,[0,1],f)$ with $f(x)=1+\epsilon x$ and with $n>2$.
\end{proof}

In the next example, we focus on the case of $4$ players in the unit interval where it is known that there exist a unique (up to permutations of the players) Nash equilibrium when the consumers are uniformly distributed. This equilibrium is such that two players choose $\frac{1}{4}$ and two players choose $\frac{3}{4}$. We prove that with a very large class of distribution of consumers, pure Nash equilibrium may not exist. It illustrates that the previous example does not rely on the choice of a particular density function. Proposition \ref{ex:4_players} is also a good illustration of the following: if $(x_1,\dots,x_n)$ is an exact equilibrium in the uniform case, then the profile $(y_1,\dots,y_n)$ where $y_k$ is the $x_k$-quantile of $f$, is not in general an equilibrium in the case where consumers are distributed according to the density $f$. The reason is that consumers, as opposed to players, do not care about the density $f$, they just shop at the closest location with respect to the distance $d$.\\

\begin{proposition}\label{ex:4_players} Consider $\mathcal{H}(4,[0,1],f)$ the Hotelling game with $4$ players on the unit interval $[0,1]$ and with a density function $f$. We denote $q_1$, $q_2$ and $q_3$ the quartiles of the distribution $f$, i.e. $\int_0^{q_1} f(x)dx=\int_{q_1}^{q_2} f(x)dx=\int_{q_2}^{q_3} f(x)dx=\int_{q_3}^{1} f(x)dx=\frac{1}{4}$. The only possible equilibrium is the configuration where two players select $q_1$ and two players select $q_3$, but this configuration is an equilibrium if and only if $q_2=\frac{q_1+q_3}{2}$. \end{proposition}

\begin{proof}
Denote $x_1,x_2,x_3,x_4$ the locations chosen by the players, and suppose without loss of generality, that $x_1 \leq x_2 \leq x_3 \leq x_4$. First, if $(x_1,x_2,x_3,x_4)$ is an equilibrium then $x_1=x_2$ and $x_3=x_4$, otherwise player 1 (resp. 4) would have a profitable deviation by playing $\frac{x_1+x_2}{2}$ (resp. $\frac{x_3+x_4}{2}$). Moreover, if $(x_1,x_2,x_3,x_4)$ is an equilibrium, the quantities $A:=\int_{0}^{x_1} f(x)dx$, $B:=\int_{x_1}^{\frac{x_1+x_3}{2}} f(x)dx$, $C:=\int_{\frac{x_1+x_3}{2}}^{x_3} f(x)dx$ and $D:=\int_{x_3}^1 f(x)dx$ have to be equal. Indeed the payoffs of players 1 and 2 are $\frac{A+B}{2}$, and the payoff of players 3 and 4 are $\frac{C+D}{2}$. Any player can have a payoff arbitrary close to $A,B,C$ or $D$ by playing, respectively $x_1-\epsilon$, $x_1+\epsilon$, $x_3-\epsilon$ or $x_3+\epsilon$ with $\epsilon$ small enough. This consideration leads to the fact that $A=B=C=D$, and therefore $x_1=x_2=q_1$, $\frac{x_1+x_3}{2}=q_2$ and $x_3=x_4=q_3$. It is then necessary that $q_2=\frac{q_1+q_3}{2}$ and in this case, the profile $x_1=x_2=q_1$ and $x_3=x_4=q_3$ is clearly an equilibrium.
\end{proof}

In the next example, we show that when the number of players is small, there may not exist any $\epsilon$-equilibrium for arbitrary small $\epsilon$.

\begin{proposition}\label{ex:no_eps_3} For any $\epsilon < \frac{1}{14}$, there is no $\epsilon$-equilibrium (additive or multiplicative) in $\mathcal{H}(3,[0,1],f)$.
\end{proposition}

\begin{proof}
We prove it for additive $\epsilon$-equilibrium first, and prove that it implies the non existence of multiplicative $\epsilon$-equilibrium.\\

We can admit without loss of generality that $\int_0^1 f(x)dx=1$. Suppose first that $x_1<x_2<x_3$. Using the notation $x_{i,i+1}:=\frac{x_i+x_{i+1}}{2}$, we have that $\int_{x_{1,2}}^{x_2} f(x)dx \leq \frac{1}{14}$ and $\int_{x_2}^{x_{2,3}} f(x)dx \leq \frac{1}{14}$, otherwise player 1 (resp. 3) would deviate to $x_2-\delta$ (resp. $x_2+\delta$) and improve his payoff by more than $\frac{1}{14}$, for arbitrary small $\delta$. It follows that $p_2(x_1,x_2,x_3) \leq \frac{1}{7}$ and:

$$\int_0^{x_{1}} f(x)dx + \int_{x_1}^{x_{1,2}} f(x)dx + \int_{x_{2,3}}^{x_3} f(x)dx+ \int_{x_3}^1 f(x)dx = 1 - \int_{x_{1,2}}^{x_2} f(x)dx - \int_{x_2}^{x_{2,3}} f(x)dx \geq \frac{6}{7}$$

We conclude that at least one of these four terms is larger than $\frac{3}{14}$. Player $2$ can have a payoff arbitrary close to any of these quantity by playing respectively $x_1-\delta$, $x_1+\delta$ ,$x_3-\delta$ or $x_3+\delta$ with a small enough $\delta$. It leads to a contradiction because $\frac{3}{14}-\frac{1}{7}=\frac{1}{14}> \epsilon$.
\\

Suppose now that $x_1=x_2=x_3$, then $p_1=p_2=p_3=\frac{1}{3}$, but player 1 can get a payoff at least arbitrary close to $\frac{1}{2}$ by playing either $x_1-\delta$ or $x_1+\delta$ for a small enough $\delta$. We have a contradiction because $\frac{1}{2}-\frac{1}{3}=\frac{1}{6} > \frac{1}{14} > \epsilon$.\\

Suppose finally that $x_1=x_2<x_3$. The same argument used in the first case gives that $\int_{x_1}^{x_{2,3}} f(x)dx  \leq \frac{1}{14}$. It is also necessary that $\int_0^{x_1} f(x)dx \leq \int_{x_1}^{x_{2,3}} f(x)dx + \frac{1}{14}$ otherwise player 1 would deviate to $x_1+\delta$ for a small enough $\delta$. Therefore $\int_{x_{2,3}}^1 f(x)dx = 1 -  \int_0^{x_{2,3}} f(x)dx \geq \frac{11}{14}$. It results that either $\int_{x_{2,3}}^{x_3} f(x)dx \geq \frac{11}{28}$ or $\int_{x_3}^1 f(x)dx \geq \frac{11}{28}$. Player 1 can have a payoff arbitrary close to these quantities by playing $x_3-\delta$ or $x_3+\delta$ for small enough $\delta$. But $p_1(x_1,x_2,x_3)=\frac{\int_{0}^{x_{2,3}} f(x)dx}{2} \leq \frac{3}{14}$, so he can improve his payoff by $\frac{11}{28}-\frac{3}{14}=\frac{5}{28}>\frac{1}{14} > \epsilon$. In any case, there is no additive $\frac{1}{14}$-equilibrium.\\

We can also prove that there is no multiplicative $\epsilon$-equilibrium. Suppose that there exists one $\boldsymbol{x}$. If $x_{dev}$ is a profile of location obtained after a possible unilateral deviation from $\boldsymbol{x}$, we have for any $k \in \{1,2,3 \}$ that $p_k(x_{dev}) \leq (1+\epsilon) p_k(\boldsymbol{x})$. Because the payoff of any player is bounded by $1$ we have:

$$p_k(x_{dev}) - p_k(\boldsymbol{x}) \leq \epsilon p_k(\boldsymbol{x}) \leq \epsilon$$

We have a contradiction because $\boldsymbol{x}$ can not be an additive $\epsilon$-equilibrium.
\end{proof}

\section{Proof of theorem \ref{thm:main_thm}}\label{se:proof}

In the current section we give a proof for Theorem \ref{thm:main_thm}. As announced the sketch of the proof in section \ref{se:result}, the extensive proof is divided in 3 steps detailed in subsections \ref{subsec:approx}, \ref{subsec:eq_step} and \ref{subsec:eps_eq}.

\subsection{The approximation of the density by step functions}\label{subsec:approx}

In this subsection, we construct a step function  $g(\epsilon_1)$ for a fixed parameter $\epsilon_1$, and we prove that $g(\epsilon_1)$ is a good approximation of $f$ when $\epsilon_1$ is small.

\begin{definition}\label{def:ell_and_ie}
For a fixed $\epsilon_1>0$, the step function $g(\epsilon_1): G \rightarrow \mathbb{R}^+$ is defined edge by edge as follows. On an edge $e \in E$, the number of step is equal to $I_e:= \left\lceil \frac{\lambda_e K}{2 \epsilon_1}\right\rceil$, and all steps are semi open intervals\footnote{However, because we only consider consumers' distributions absolutely continuous with respect to the Lebesgue measure, the definition of $g$ on a singleton is not relevant.}of the same  length:

\begin{equation}\label{def:ell}
\ell(\epsilon_1):=\frac{\lambda_e}{\left\lceil \frac{\lambda_e K}{2 \epsilon_1}\right\rceil}
\end{equation}
It is clear that $\frac{\lambda_e}{\ell_e(\epsilon_1)}= I_e \in \mathbb{N}$. The function $g(\epsilon_1)$ is constant on a step, and its common value is equal to the value of $f$ at the middle point of the step. More precisely, if we denote $g_e(\epsilon_1)$ and $f_e$ is the restrictions of $g(\epsilon_1)$ and $f$ on the edge $e$ and if we use the natural identification between an edge $e$ and the real interval $[0,\lambda_e]$,\footnote{We use also the arbitrary orientation on the edges. The choice of this orientation does not play any role.}we have that:\\

$\left\{
\begin{array}{lll}
\text{First step: } (i=0)& \forall x \in [0,\ell_e(\epsilon_1)), & g_e(x)=f_e(\frac{\ell_e(\epsilon_1)}{2})
~~\\
...
 ~~\\
%\forall i \in \{1,\dots,\frac{\lambda_e K}{\epsilon_1} \}, ~~ 
i^{th} \text{step: }& \forall x \in [i \ell_e(\epsilon_1),(i+1) \ell_e(\epsilon_1)),  & g_e(x)=f_e( (i+\frac{1}{2})\ell_e(\epsilon_1) )
  ~~\\
...
 ~~\\
\text{Last step: } (i=I_e-1)& \forall x \in [\lambda_e-\ell_e(\epsilon_1),\lambda_e), & g_e(x)=f_e(\lambda_e-\frac{\ell_e(\epsilon_1)}{2})\\
\end{array}
\right.$
\end{definition}

The next lemma shows that $g(\epsilon_1)$ is a good approximation of $f$ when $\epsilon_1$ is small enough. 

\begin{lemma}\label{le:f-g}
$|| f- g(\epsilon_1) ||_{\infty} \leq \epsilon_1$
\end{lemma}

\begin{proof}
Let $x \in G$. Let $e,i$ such that $x \in e$ and $x \in [i \ell_e(\epsilon_1),(i+1) \ell_e(\epsilon_1)[$. Using the definition of $g(\epsilon_1)$ and the fact that $f$ is $K$-Lipschitz, we have:

\begin{align*}
|f(x)-g(x)| \leq & |f_e(x)-f_e( (i+\frac{1}{2}) i \ell_e(\epsilon_1))|
+ |f_e( (i+\frac{1}{2}) i \ell_e(\epsilon_1)) -g_e(x)| \\ 
\leq & K | x - (i+\frac{1}{2}) \ell_e(\epsilon_1) | +0\\
\leq &  K \frac{\ell_e(\epsilon_1)}{2} = K \frac{\lambda_e}{2\left\lceil \frac{\lambda_e K}{2 \epsilon_1}\right\rceil} \leq \epsilon_1
\end{align*} 
\end{proof}

The payoff function in the game $\mathcal{H}(n,S,f)$ is defined in equation (\ref{eq:payoff_function}). In subsection \ref{subsec:eq_step} we study the game $\mathcal{H}(n,S,g(\epsilon_1))$ where the consumers are distributed according to the density $g(\epsilon_1)$. The payoff function in this game is denote $\pi$ and is now formally defined.

\begin{definition}
We define the payoff function $\pi:=(\pi_1,\dots,\pi_n)$ in the game $\mathcal{H}(n,S,g(\epsilon_1))$. For a given profile of action $\boldsymbol{x}$, we have that:

\begin{equation}\label{eq:payoff_g}
\pi_{i}(\boldsymbol{x}) := \frac{1}{\card(\{j\in \{1,\dots,n \}: x_{j}=x_{i}\})}\sum_{A \subset \{x_1,\dots,x_n \} | x_{i}\in A } \frac{\tilde{\mu}(D_A)}{\card(A)}.
\end{equation}
where $D_A$ is defined in equation (\ref{def:C_A}) and where $\tilde{\mu}(D_A)= \int_{D_A} g(\epsilon_1)(x) d\mathcal{L}(x)$. 
\end{definition}

\subsection{The equilibrium with the step function density $g(\epsilon_1)$}\label{subsec:eq_step}

In this subsection, we provide a method to construct an (exact) equilibrium in the game with density function $g(\epsilon_1)$ and with a large enough number of players.\\

In subsection \ref{subsubse:description} we describe a profile of location $\boldsymbol{x}(\theta,\epsilon_1)$ for some fixed real numbers $\theta>0$ and $\epsilon_1>0$. Lemma \ref{le:nb_of_player_in_x} below is useful to compute the number $h(\theta,\epsilon_1)$ of players in such a configuration. In subsection \ref{subsubse:properties_of_x}, we prove that $\boldsymbol{x}(\theta,\epsilon_1)$ is an exact equilibrium in the game $\mathcal{H}(h(\theta,\epsilon_1),S,g(\epsilon_1))$. Before describing the profile $\boldsymbol{x}(\theta,\epsilon_1)$ we emphasize on its differences with existing constructions.\\

The method used to construct the auxiliary equilibrium in \ref{subsubse:description} is inspired by the ones used in \cite{palvolgyi11} and \cite{fournier14}. Nevertheless, the main difference is that we provide here a construction that depends on the steps of the function $g(\epsilon_1)$, as opposed to the mentioned constructions that describe edge-dependent models. The current construction depends on two parameters $\theta$ and $\epsilon_1$, and the lengths of the intervals between players depend on the value of the function $g(\epsilon_1)$. Note that in the papers mentioned above, it is not trivial to add "fake" vertices of degree 2 and to consider the steps of $g$ as artificial edges between such vertices. The existing constructions depends strongly on the number of vertices, and vertices of degree 2 are not taken into account: in particular it is not a necessary condition that for a large number of players, at least one player chose to play in a vertex with degree $2$, as opposed to vertices with degree at least $3$.

\subsubsection{Description of the profile $\boldsymbol{x}(\theta,\epsilon_1)$} \label{subsubse:description}

We now describe the profile $\boldsymbol{x}(\theta,\epsilon_1)$ for fixed parameters $\theta>0$ and $\epsilon_1>0$. In the definition of the function $g(\epsilon_1)$ in subsection \ref{subsec:approx}, we divide each edge $e \in E$ into $I_e$ steps, on which $g(\epsilon_1)$ is constant. It is important that every edge contains at least two steps, therefore in the sequel of the paper we always suppose $\epsilon_1<\frac{\min \lambda_e K}{2}$. We distinguish 3 different types of steps: the steps adjacent to a leaf, i.e. a vertex with degree $1$, the steps adjacent to a vertex of degree at least 3,\footnote{Remember we can assume without loss of generality that all vertices have degree different from 2, as explained in subsection \ref{subse:network}.} and the remaining steps, that we call interior steps. Steps of the same type have the same players' positioning, that we now descibe.\\

\begin{itemize}
\item[1/]  We first consider an interior step, that is not at an extremity of the edge. Denote $e$ the edge such that this step belongs to $e$ and $i \in \{1,\dots,I_e-2 \}$ such that the step belongs to the $i$-th step of the edge $e$ (i.e. it can be identified with the real interval $[i \ell_e(\epsilon_1),(i+1) \ell_e(\epsilon_1)]$). The value of $g$ on this step is $ g_e^i(\epsilon_1)=f_e((i+\frac{1}{2}) \ell_e(\epsilon_1))$. For simplicity, we denote it $g_e^i$.\\

There are a total of $\left\lceil \frac{g_e^i \ell_e (\epsilon_1)}{2  \theta} \right\rceil$ players in the interior of the edge and two players on each extremity of the step. At a distance $\frac{2 \theta}{g_e^i}$ from one of the boundary there are two players. All the other players are single, at distance $\frac{a_e^i \theta}{g_e^i}$ from each other and at distance $\frac{2 \theta}{g_e^i}$ from the coupled players, where $a_e^i$ is equal to:

\begin{equation}\label{def_alpha}
a_e^i=  \frac{ \frac{g_e^i \ell_e(\epsilon_1)}{\theta}-6 }{ \left\lceil \frac{ g_e^i \ell_e(\epsilon_1)}{2 \theta} \right\rceil -3 }
\end{equation}

\begin{figure}[H]

\centering

\begin{tikzpicture}[>=latex]

    \draw[-] (0,0) --(11,0);

    \node[] at (0,0) {$|$};
     \node[above=8pt] at (0,0) {$2$};
     
    \node[] at (1.5,0) {$|$};
     \node[above=8pt] at (1.5,0) {$2$};
     
   \node[] at (3,0) {$|$};
   \node[above=8pt] at (3,0) {$1$};
   
    \node[] at (4,0) {$|$};
    \node[above=8pt] at (4,0) {$1$};

       \node[] at (8.5,0) {$|$};
\node[above=8pt] at (8.5,0) {$1$};

    \node[] at (9.5,0) {$|$};
\node[above=8pt] at (9.5,0) {$1$};    
    
     \node[] at (11,0) {$|$};
\node[above=8pt] at (11,0) {$2$};

     \node[above=8pt] at (5,0) {...};
    \node[above=8pt] at (7.5,0) {...};

  \draw[|<->|] (0,-0.5) -- (1.5,-0.5);    
  \node[below=4pt] at (0.75,-0.5) {$\frac{2 \theta}{g_e^i}$};
  
  \draw[|<->|] (1.5,-0.5) -- (3,-0.5);    
  \node[below=4pt] at (2.25,-0.5) {$\frac{2 \theta}{g_e^i}$};
 
   \draw[|<->|] (3,-0.5) -- (4,-0.5);    
 \node[below=4pt] at (3.5,-0.5) {$\frac{a_e^i \theta}{g_e^i}$};

    \draw[|<->|] (8.5,-0.5) -- (9.5,-0.5);    
  \node[below=4pt] at (9,-0.5) {$\frac{a_e^i \theta}{g_e^i}$};

  \draw[|<->|] (9.5,-0.5) -- (11,-0.5);    
  \node[below=4pt] at (10.25,-0.5) {$\frac{2 \theta}{g_e^i}$};
 
    \end{tikzpicture}

\caption{\label{fig:desc_type_1} Players on a step of type 1 (interior step).}
\end{figure}
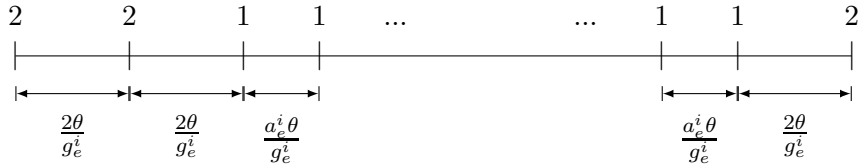

The previous description is compatible with the fact that the total length of the step is $\ell_e(\epsilon_1)$, because the parameter $a_e^i$ is such that the following equation holds:
$$\ell_e(\epsilon)=\frac{6\theta}{g_e^i}+ \frac{a_e^i \theta}{g_e^i} \times (\left\lceil \frac{g_e^i \ell_e (\epsilon_1)}{2  \theta} \right\rceil -3)$$

\item[2/] We consider now a step connected to a vertex $v$ with degree at least 3. There are $\left\lceil \frac{g_e^i \ell_e (\epsilon_1)}{2  \theta} \right\rceil$ players in the interior of this edge, plus $deg(v)$ players in the vertex $v$ and two players on the other extremity. There are 2 players at a distance $\frac{2\theta}{g_e^i}$ from $v$. All the other players are single, at distance $\frac{a_e^i \theta}{g_e^i}$ from each other and at distance $\frac{2 \theta}{g_e^i}$ from coupled players, where $a_e^i$ was defined in equation (\ref{def_alpha}).

\begin{figure}[H]

\centering

\begin{tikzpicture}[>=latex]

    \draw[-] (0,0) --(11,0);
    
    \node[] at (0,0) {$|$};
    \node[above=4pt] at (0,0) {$deg(v)$};

    \node[] at (1.5,0) {$|$};
    \node[above=8pt] at (1.5,0) {$2$};
    
    \node[] at (3,0) {$|$};
    \node[above=8pt] at (3,0) {$1$};
    
    \node[] at (4,0) {$|$};
    \node[above=8pt] at (4,0) {$1$};
    
       \node[above=8pt] at (5,0) {...};

       \node[above=8pt] at (7.5,0) {...};
     
             \node[] at (8.5,0) {$|$};
    \node[above=8pt] at (8.5,0) {$1$};

           \node[] at (9.5,0) {$|$};
    \node[above=8pt] at (9.5,0) {$1$};
    
     \node[] at (11,0) {$|$};
\node[above=8pt] at (11,0) {$2$};     
     
  \draw[|<->|] (0,-0.5) -- (1.5,-0.5);    
  \node[below=4pt] at (0.75,-0.5) {$\frac{2 \theta}{g_e^i}$};
  
  \draw[|<->|] (1.5,-0.5) -- (3,-0.5);    
  \node[below=4pt] at (2.25,-0.5) {$\frac{2 \theta}{g_e^i}$};
  
       \draw[|<->|] (3,-0.5) -- (4,-0.5);    
  \node[below=4pt] at (3.5,-0.5) {$\frac{a_e^i \theta}{g_e^i}$};

 \draw[|<->|] (8.5,-0.5) -- (9.5,-0.5);    
  \node[below=4pt] at (9,-0.5) {$\frac{a_e^i \theta}{g_e^i}$};    
    
  \draw[|<->|] (9.5,-0.5) -- (11,-0.5);    
  \node[below=4pt] at (10.25,-0.5) {$\frac{2 \theta}{g_e^i}$};

    \end{tikzpicture}

\caption{\label{fig:desc_type_2} Players on a step of type 2 (connected to a vertex $v$).}
\end{figure}
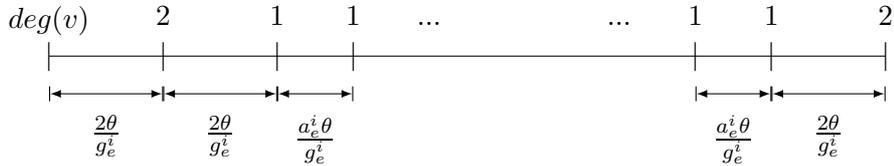

The previous description is also compatible with the fact that the total length of the step is $\ell_e(\epsilon_1)$, because the parameter $a_e^i$ is such that the following equation holds:
$$\ell_e(\epsilon)=\frac{6\theta}{g_e^i}+ \frac{a_e^i \theta}{g_e^i} \times (\left\lceil \frac{g_e^i \ell_e \epsilon_1)}{2  \theta} \right\rceil-3) $$

\item[3/] We now consider a step connected to a leaf. In the interior of the step there are $\left\lceil \frac{g_e^i \ell_e (\epsilon_1)}{2  \theta} \right\rceil+2$ players and one player is located at the extremity opposed to the the leaf. Two players are at distance $\frac{\theta}{g_e^i}$ and two players are at distance $\frac{3\theta}{g_e^i}$. The other players are single at distance $\frac{b_e^i \theta}{g_e^i}$ from each other, where:

\begin{equation}
b_e^i=  \frac{ \frac{g_e^i \ell_e(\epsilon_1)}{\theta}-7 }{ \left\lceil \frac{ g_e^i \ell_e(\epsilon_1)}{2 \theta} \right\rceil - 3 }
\end{equation}
\end{itemize}

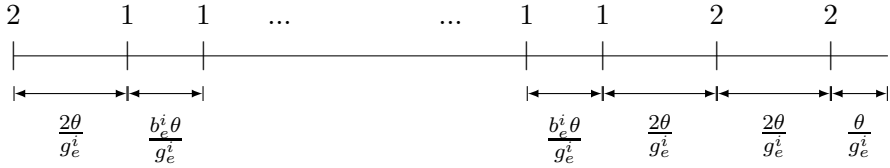
\begin{figure}[H]

\centering

\begin{tikzpicture}[>=latex]

    \draw[-] (0.5,0) --(12,0);
    \node[] at (2,0) {$|$};
    \node[] at (3,0) {$|$};

    \node[] at (0.5,0) {$|$};
   % \node[] at (3.5,0) {$|$};
    
  %  \node[] at (6.25,0) {$|$};
    \node[] at (7.25,0) {$|$};
    
    \node[] at (8.25,0) {$|$};
    \node[] at (9.75,0) {$|$};
    \node[] at (11.25,0) {$|$};
    
    \node[above=8pt] at (0.5,0) {$2$};
    \node[above=8pt] at (2,0) {$1$};
    \node[above=8pt] at (3,0) {$1$};
 %   \node[above=8pt] at (2.5,0) {$1$};
 %   \node[above=8pt] at (3.5,0) {$1$};
    \node[above=8pt] at (4,0) {...};
     \node[above=8pt] at (6.25,0) {...};
  %  \node[above=8pt] at (6.25,0) {$1$};
    \node[above=8pt] at (7.25,0) {$1$};
    \node[above=8pt] at (8.25,0) {$1$};
    \node[above=8pt] at (9.75,0) {$2$};
    \node[above=8pt] at (11.25,0) {$2$};
    
%    \node[above=8pt] at (14,0) {Number of players};
    
  \draw[|<->|] (0.5,-0.5) -- (2,-0.5);    
  \node[below=4pt] at (1.25,-0.5) {$\frac{2 \theta}{g_e^i}$};
  
  \draw[|<->|] (2,-0.5) -- (3,-0.5);    
  \node[below=4pt] at (2.5,-0.5) {$\frac{b_e^i \theta}{g_e^i}$};

%  \draw[|<->|] (1.5,-0.5) -- (2.5,-0.5);    
%  \node[below=4pt] at (2,-0.5) {$\frac{b_e^i \theta}{g_e^i}$};
  
%  \draw[|<->|] (2.5,-0.5) -- (3.5,-0.5);    
%  \node[below=4pt] at (3,-0.5) {$\frac{b_e^i \theta}{g_e^i}$};
  
%   \draw[|<->|] (6.25,-0.5) -- (7.25,-0.5);    
 % \node[below=4pt] at (6.75,-0.5) {$\frac{b_e^i \theta}{g_e^i}$};
  
   \draw[|<->|] (7.25,-0.5) -- (8.25,-0.5);    
  \node[below=4pt] at (7.75,-0.5) {$\frac{b_e^i \theta}{g_e^i}$};

    \draw[|<->|] (8.25,-0.5) -- (9.75,-0.5);    
  \node[below=4pt] at (9,-0.5) {$\frac{2 \theta}{g_e^i}$};
  
  \draw[|<->|] (9.75,-0.5) -- (11.25,-0.5);    
  \node[below=4pt] at (10.5,-0.5) {$\frac{2 \theta}{g_e^i}$};

  \draw[|<->|] (11.25,-0.5) -- (12,-0.5);    
  \node[below=4pt] at (11.625,-0.5) {$\frac{ \theta}{g_e^i}$};
 
    \end{tikzpicture}

\caption{\label{fig:desc_type_3} Players in a step of type 3 (connected to a leaf).}
\end{figure}

The previous description is, once again, compatible with the fact that the total length of the step is $\ell_e(\epsilon_1)$, because the parameter $b_e^i$ is such that the following equation holds:
$$\ell_e(\epsilon)=\frac{7\theta}{g_e^i}+ \frac{b_e^i \theta}{g_e^i} \times (\left\lceil \frac{g_e^i \ell_e (\epsilon_1)}{2  \theta} \right\rceil-3) $$

We described the players' positioning on all the different types of steps. We now compute how many players are located on the graph in the described profile, with respect to the parameters $\theta$ and $\epsilon_1$.

\begin{lemma}\label{le:nb_of_player_in_x}
The total number of players in the profile $\boldsymbol{x}(\theta,\epsilon_1)$ is: 
$$h(\theta,\epsilon_1) := \sum_{e \in E} \sum_{i=0}^{I_e-1} \left\lceil \frac{g_e^i(\epsilon_1) \ell_e(\epsilon_1)}{2 \theta} \right\rceil + 2 \card(E_L)+ 2\sum_{e \in E} \left\lceil \frac{\lambda_e K}{2 \epsilon_1} \right\rceil  $$ where $E_L$ denotes the set of edges $(u,v)$ with $deg(u)=1$ or $deg(v)=1$.
\end{lemma}

\begin{proof}
We denote $g_e^i$ instead of $g_e^i(\epsilon_1)$. In the interior of every step that is not connected to a leaf (types $1$ and $2$) there are $\left\lceil \frac{g_e^i \ell_e(\epsilon_1)}{2 \theta} \right\rceil$ players. In the interior of every step connected to a leaft (type $3$) there are $\left\lceil \frac{g_e^i \ell_e(\epsilon_1) }{2 \theta} \right\rceil +2$ players and there are $\card(E_L)$ such steps. Because the number of steps in an edge $e$ is $I_e$, there are in total a number of players in the interior of steps equal to: $$\sum_{e \in E} \sum_{i=0}^{I_e-1} \left\lceil \frac{g_e^i \ell_e(\epsilon_1)}{2 \theta} \right\rceil + 2 \card(E_L)$$

Moreover, in the intersection of two steps in the same edge there are $2$ player. In the edge $e$, there are $\frac{\lambda_e}{\ell_e(\epsilon_2)}$ steps, and then $\frac{\lambda_e}{\ell_e(\epsilon_2)}-1$ such intersections. In a vertex $v$ there are $deg(v)$ players. Using that $\sum_{v \in V} deg(v)=2 \card(E)$ and the definition of $\ell_e(\epsilon_1)$, we have that the total number of players located at the intersections between steps is equal to:

$$2 \card(E) + 2 \sum_{e \in E}\left(\frac{\lambda_e}{\ell_e(\epsilon_2)}-1\right)= 2 \sum_{e \in E}\frac{\lambda_e}{\ell_e(\epsilon_2)}= 2 \sum_{e \in E} \left\lceil \frac{\lambda_e K}{2 \epsilon_1} \right\rceil $$

The sum of these two quantities is equal to the announced $h(\theta,\epsilon_1)$.
\end{proof}

\subsubsection{The profile $\boldsymbol{x}(\theta,\epsilon_1)$ is an equilibrium in $\mathcal{H}(h(\theta,\epsilon_1),S,g(\epsilon_1))$}\label{subsubse:properties_of_x}

In \ref{subsubse:description} we described the profile $\boldsymbol{x}(\theta,\epsilon_1)$ for fixed parameters $\theta>0$ and $\epsilon_1 \in (0,\frac{\min \lambda_e K}{2})$ and proved that it contains $h(\theta,\epsilon_1)$ players. We now prove that it is an (exact) equilibrium in the game $\mathcal{H}(h(\theta,\epsilon_1),S,g(\epsilon_1))$.\\

\begin{proposition}\label{prop:eq_with_g}
For any $\epsilon_1 \in (0,\frac{\min \lambda_e K}{2})$, if the number of player $n$ is larger than

\begin{equation}\label{def:M_esp_1}
\Omega(\epsilon_1):= 5 \card(E) +  \left\lceil \frac{5 L (M+\epsilon_1)}{(m-\epsilon_1)}\left(\frac{K}{2 \epsilon_1} + \frac{1}{\min \lambda_e}\right) +   \frac{3 L K}{2 \epsilon_1} \right\rceil
\end{equation} 
then there is a pure Nash equilibrium in $\mathcal{H}(n,S,g(\epsilon_1))$.
\end{proposition}

In the proof of Proposition \ref{prop:eq_with_g}, we construct a  pure Nash equilibrium that is a variation of the profile $\boldsymbol{x}(\theta,\epsilon_1)$, for a good parameter $\theta$. Before proving Proposition \ref{prop:eq_with_g}, we state and prove two useful lemma.

\begin{lemma}\label{le:4_claims}
Let $\epsilon_1 \in (0,\frac{\min \lambda_e K}{2})$ and suppose that the parameter $\theta$ satisfies $\theta \leq  \frac{g_e^i(\epsilon_1) \ell(\epsilon_1)}{10}$ for all $e \in E$ and all $i \in \{ 0, I_e-1\}$. Then the profile $\boldsymbol{x}(\theta,\epsilon_1)$ is such that we have the following statements:\\
\begin{itemize}
\item Claim 1: For all $e \in E$ and all $i \in \{ 0, I_e-1\}$ we have that $1 \leq a_e^i \leq 2$ and that $1 \leq b_e^i \leq 2$.\\

\item Claim 2: Every player that shares its location with at least one other player has a payoff equal to $\theta$.\\

\item Claim 3: The payoff of every player is at least $\theta$ and at most $2\theta$.\\

\item Claim 4: Because there are players located in every vertex, the part of the network between two neighbors' locations is always a segment, that can be identified with a real interval. No such interval contains a larger quantity of consumers than $2\theta$.\\
\end{itemize}

\end{lemma}

\begin{proof}
\noindent \underline{Proof of Claim 1:} Let $e \in E$ and $i \in \{ 0, I_e-1\}$. Remember that $a_e^i$ and $b_e^i$ are defined as:

$$a_e^i=  \frac{ \frac{g_e^i \ell_e(\epsilon_1)}{\theta}-6 }{ \left\lceil \frac{ g_e^i \ell_e(\epsilon_1)}{2 \theta} \right\rceil - 3 }$$

$$b_e^i=  \frac{ \frac{g_e^i \ell_e(\epsilon_1)}{\theta}-7 }{ \left\lceil \frac{ g_e^i \ell_e(\epsilon_1)}{2 \theta} \right\rceil - 3 }$$

We obviously have $b_e^i < a_e^i$. First remark that:
$$ \frac{g_e^i \ell_e(\epsilon_1)} {\theta}-6  = 2 \frac{g_e^i \ell_e(\epsilon_1)} {2 \theta}-6 \leq 2 \left\lceil \frac{ g_e^i \ell_e(\epsilon_1)}{2 \theta} \right\rceil - 6  = 2  (\left\lceil \frac{ g_e^i \ell(\epsilon_1)}{2 \theta} \right\rceil - 3 )$$
And it follows that $a_e^i \leq 2$. Moreover, because $\theta \leq  \frac{g_e^i \ell_e(\epsilon_1)}{10}$, we have that $5 \leq \frac{g_e^i \ell_e(\epsilon_1)}{2\theta}$. It implies that the second inequality below:

$$  \left\lceil \frac{g_e^i \ell_e(\epsilon_1) }{2 \theta} \right\rceil -3 \leq \frac{g_e^i \ell_e(\epsilon_1) }{2 \theta}-2  \leq \frac{g_e^i \ell_e(\epsilon_1)}{ \theta} -7 $$

And it follows that $1 \leq b_e^i$.\\

\noindent \underline{Proof of Claim 2:} It follows from the definition of $\boldsymbol{x}(\theta,\epsilon_1)$ that when $k \geq 2$ players are in the same location their payoffs are equal to $\frac{1}{k}$ times the sum of $k$ terms, each equal to $\frac{\theta}{g_i^e}\times g_e^i$ for a given $e \in E$ and a given $i \in \{1,\dots,I_e-1\}$ (on each direction, the location attracts an interval of consumers of length $\frac{\theta}{g_i^e}$, and the density of consumer is equal to $g_e^i$).\\

\noindent \underline{Proof of Claim 3:} Using Claim 2, we just have to prove it for single players. Depending on the type of the step this player belongs to and his position on the step, his payoff can be equal to $\frac{a_e^i \theta}{g_e^i}\times g_e^i$ or to $\frac{a_e^i \theta}{2g_e^i}\times g_e^i + \frac{2 \theta}{2g_e^i}\times g_e^i$ (in type 1 or 2) or to $\frac{b_e^i \theta}{g_e^i}\times g_e^i$ or to $\frac{b_e^i \theta}{2g_e^i}\times g_e^i + \frac{2 \theta}{2g_e^i}\times g_e^i$ (type 3), for some $e \in E$ and $i,i' \in \{1,\dots,I_e-1\}$. Using Claim 1, we have that $1 \leq b_e^i < a_e^i \leq 2$, and it results that his payoff belongs to $[\theta,2\theta]$.\\

\noindent \underline{Proof of Claim 4:} This claim is obvious since the largest length of such an interval is $\frac{2\theta}{g_e^i}$ and the density is $g_e^i$, for some $e \in E$ and $i \in \{1,\dots,I_e-1\}$.\\
\end{proof}

Lemma \ref{le:4_claims} is the toolbox to prove the next lemma.

\begin{lemma}\label{le:xi_small_equilibrium}
Suppose that $\epsilon_1 \in (0,\frac{\min \lambda_e K}{2})$ and that the parameter $\theta$ satisfies $\theta \leq  \frac{g_e^i(\epsilon_1) \ell(\epsilon_1)}{10}$ for all $e \in E$ and all $i \in \{ 0, I_e-1\}$. Then the profile $\boldsymbol{x}(\theta,\epsilon_1)$ is an equilibrium of $\mathcal{H}(h(\theta,\epsilon_1),S,g(\epsilon_1))$.
\end{lemma}

\begin{proof}
We will make an intensive use of Lemma \ref{le:4_claims}, that applies because $\theta \leq  \frac{g_e^i(\epsilon_1) \ell(\epsilon_1)}{10}$. Lemma \ref{le:nb_of_player_in_x} states that $\boldsymbol{x}(\theta,\epsilon_1)$ involves $h(\theta,\epsilon_1)$ players so it is a possible profile of actions in this game. For simplicity, we denote during this proof $m:=h(\theta,\epsilon_1)$, $g_e^i$ instead of $g_e^i(\epsilon_1)$ and $\boldsymbol{x}$ or $(x_1,\dots,x_m)$ the profile $\boldsymbol{x}(\theta,\epsilon_1)$. We now conclude that $\boldsymbol{x}$ is an equilibrium of $\mathcal{H}(m,S,g({\epsilon_1}))$. Consider any player $k$ in the profile $\boldsymbol{x}$ and denote $x_k$ his location. We study all possible deviations and conclude that none of them are profitable.\\

\underline{Case 1:} Suppose first that the player $k$ shares his position with at least one other player in the profile $\boldsymbol{x}$. This is the most simple case: the domain of attraction of $x_k$ does not change after the deviation of player $k$, because there is still at least one player left in $x_k$. Notice that because there is no tie on set with strictly positive measure in $\boldsymbol{x}$, the domain of attraction of $y$ is the set of consumers that shop at location $y$. Claim 1 of Lemma \ref{le:4_claims} assures that $p_k(\boldsymbol{x})$, the payoff of player $k$, is equal to $\theta$. We now consider his possible unilateral deviation to a location $y \in G$.  \\

\textit{Case 1-a:} If player $k$ deviates to a location $x_j$ that was chosen by exactly one player $j$ in the profile $\boldsymbol{x}$, we have:

$$\pi_k(x_1,\dots,x_{k-1},x_j,x_{k+1},\dots,x_m)=\frac{\pi_j(\boldsymbol{x})}{2} \leq \frac{2 \theta}{2} = \pi_k(\boldsymbol{x})$$

where the inequality comes from Claim 3 in \ref{le:4_claims}. Such a deviation is not profitable.\\

\textit{Case 1-b:} Suppose that player $k$ deviates to a location $y$ that was chosen by at least two players in the profile $\boldsymbol{x}$. Using Claim 2 in \ref{le:4_claims} and the fact that players in $y$ now have to share the domain of attraction with one more player, we have that:

$$\pi_k(x_1,\dots,x_{k-1},y,x_{k+1},\dots,x_m) < \theta = \pi_k(\boldsymbol{x})$$

The deviation is not profitable.\\

\textit{Case 1-c:} Suppose finally that player $k$ deviates in a location $y$ that belongs to an interval $[x_j,x_{j'}]$ between two neighbors $x_j$ and $x_{j'}$. It follows from the definition of $\boldsymbol{x}$ that the density $g(\epsilon_1)$ is constant on this interval. Therefore, by deviating to $y \in [x_j,x_{j'}]$ player $k$ attracts half the consumers located in $[x_j,x_{j'}]$. Using Claim 4 in Lemma \ref{le:4_claims} we have that:

$$\pi_k(x_1,\dots,x_{k-1},y,x_{k+1},\dots,x_m) \leq \frac{2 \theta}{2} = \pi_k(\boldsymbol{x})$$

And we proved that there is no profitable deviation for such a player.\\

\underline{Case 2:} Suppose now that the player $k$ doesn't share his location with another player in the profile $\boldsymbol{x}$. Claim 1 of Lemma \ref{le:4_claims} states that his payoff $p_k(\boldsymbol{x})$ is at least $\theta$. This second case is a bit more delicate because when the player $k$ deviates, it changes the domain of attraction of his neighbors. Without loss of generality we denote $x_{k-1}$ and $x_{k+1}$ the neighbors' location of $x_k$ ( a location with a single player always have left and right neighbor in $\boldsymbol{x}$). We denote $p_k(\boldsymbol{x}):=\frac{(x_{k+1}-x_{k-1})g_e^i}{2}$ his payoff. If the player $k$ deviates to a location that does not belongs to the interval $[x_{k-1},x_{k+1}]$ the situation is similar to the one considered in case 1.\\ 

\textit{Case 2-a:} Suppose first that the player $k$ deviates to a location $y$ in the interior of the interval $[x_{k-1},x_{k+1}]$. It follows from the definition of $\boldsymbol{x}$ that the density $g(\epsilon_1)$ is constant on this interval. After such a deviation his payoff would still be equal to $\frac{(x_{k+1}-x_{k-1})g_e^i}{2}$.\\

\textit{Case 2-b:} Suppose now that the player $k$ deviates to the location of one of his neighbors. Several sub-cases have to be considered. Suppose first that player $k$ is at distance $\frac{4 \theta}{g_e^i}$ from the left extremity of a step of type 1 or 2, see figure \ref{fig:desc_type_1} or \ref{fig:desc_type_2}. His payoff in this case satisfies:

$$ \pi_k(\boldsymbol{x}) = \left(\frac{\theta}{g_e^i}+\frac{a_e^i \theta}{2g_e^i}\right)g_e^i=\theta+\frac{a_e^i \theta}{2}$$

If such a player $k$ deviates to his right neighbor's location $y$ (i.e. at distance $\frac{4 \theta}{g_e^i}+\frac{a_e^i \theta}{g_e^i}$ from the left extremity), his payoff after deviation would satisfy:

$$ \pi_k(x_1,\dots,x_{k-1},y,x_{k+1},\dots,x_m) = \frac{\frac{2\theta+a_e^i \theta}{2 g_e^i}+\frac{a_e^i \theta}{2 g_e^i}}{2}g_e^i = \frac{\theta + a_e^i \theta}{2 } \leq \pi_k(\boldsymbol{x}) $$

Such a deviation is not profitable. If player $k$ deviates to his left neighbor's location $y$ (i.e. at distance $\frac{2 \theta}{g_e^i}$ from the left extremity), his payoff after deviation would satisfy:

$$\pi_k(x_1,\dots,x_{k-1},y,x_{k+1},\dots,x_m) = \frac{\frac{\theta}{g_e^i}+\frac{\theta}{g_e^i}+\frac{a_e^i \theta}{2 g_e^i}}{3} g_e^i = \frac{2 \theta}{3}+\frac{a_e^i \theta}{6}< \pi_k(\boldsymbol{x})  $$

%$$\pi_k(x_1,\dots,x_{k-1},y,x_{k+1},\dots,x_m) = \frac{(deg(v)-1)\theta+\frac{2\theta+a_e^i \theta}{2 g_e^i}}{deg(v)+1}g_e^i \leq \frac{(deg(v)+1) \theta}{deg(v)+1} =\pi_k(\boldsymbol{x})$$

Such a deviation is not profitable. The two previous arguments are also valid when we consider the player at distance $\frac{5\theta}{g_e^i}$ from the right extremity of a step of type 3 (see figure \ref{fig:desc_type_3}).

The last sub-case we have to consider is the case of a single player that is in a location at distance $\frac{a_e^i \theta}{g_e^i}$ (resp. $\frac{b_e^i \theta}{g_e^i}$) from his two neighbors, in an edge of type 1 or 2 (resp. 3). His payoff $p_k(\boldsymbol{x})$ is equal to $\frac{a_e^i \theta}{g_e^i}g_e^i=a_e^i \theta$ (resp. $b_e^i \theta$). If he deviates to one of his neighbor location that is also a player at distance $\frac{a_e^i \theta}{g_e^i}$ (resp. $\frac{b_e^i \theta}{g_e^i}$) from his two neighbors, the payoff after deviation would satisfy:

$$\pi_k(x_1,\dots,x_{k-1},y,x_{k+1},\dots,x_m) = \frac{\frac{3 a_e^i \theta}{2g_e^i}g_e^i}{2}\leq \frac{3}{4}\pi_k(\boldsymbol{x})$$

(resp. $\frac{\frac{3 b_e^i \theta}{2g_e^i}g_e^i}{2}\leq \frac{3}{4}\pi_k(\boldsymbol{x})$) Such a deviation is not profitable. If he deviates to his neighbor location that is at distance $\frac{a_e^i \theta}{g_e^i}$ (resp. $\frac{b_e^i \theta}{g_e^i}$) from one neighbor and at distance $\frac{2 \theta}{g_e^i}$ from his other neighbor, the payoff after deviation would satisfy:

$$\pi_k(x_1,\dots,x_{k-1},y,x_{k+1},\dots,x_m) = \frac{2 a_e^i \theta + 2 \theta}{4} = \frac{a_e^i \theta}{2}+\frac{\theta}{2} \leq \pi_k(\boldsymbol{x})$$

(resp. $\frac{2 b_e^i \theta + 2 \theta}{4} = \frac{b_e^i \theta}{2}+\frac{\theta}{2} \leq \pi_k(\boldsymbol{x})$ Such a deviation is not profitable. All possible cases were analyzed, and it follows that there is no possible deviation from $\boldsymbol{x}$, which is an equilibrium in $\mathcal{H}(m,S,g(\epsilon_1))$.\end{proof}

We are now ready to prove Proposition \ref{prop:eq_with_g} by modifying the profile $\boldsymbol{x}(\theta,\epsilon_1)$ with an adequate $\theta$.

\begin{proof}{of Proposition \ref{prop:eq_with_g}.\\}
This is a two steps proof. First we show that for a given $\theta \in \mathbb{R}$, the profile $\boldsymbol{x}(\theta,\epsilon_1)$ is a pure Nash equilibrium in the game $\mathcal{H}(n',S,g(\epsilon_1))$ with a number of players $n'$ such that $n \leq n' \leq n+\sum_e I_e$. Second, we prove that from this equilibrium profile, we can derive another profile of location $\hat{\boldsymbol{x}}(\epsilon_1)$ that is an equilibrium in $\mathcal{H}(n,S,g(\epsilon_1))$. Indeed, in the profile $\boldsymbol{x}(\theta,\epsilon_1)$ there are at least $\sum_e I_e$ players that are unnecessary: the same configuration without them is an equilibrium in the game with less players. We can therefore obtain an equilibrium in the game $\mathcal{H}(n,S,g(\epsilon_1))$ with exactly $n$ players. A similar technic of going thought an auxiliary equilibrium with extra unnecessary players was already used in \citep{fournier14}.\\

\textbf{\underline{First step:}} The function that maps $\theta \in \mathbb{R}^+$ to $h(\theta,\epsilon_1) \in \mathbb{N}$ is not necessarily onto. Nevertheless it is a sum of a constant and $\sum_{e \in E} I_e$ terms of the form $\left\lceil \frac{c}{\theta} \right\rceil$ with a constant $c$. Each term of this form is decreasing, right continuous and has jumps of amplitude 1, i.e. for every $\theta$:

$$ \lim\limits_{\substack{\theta \to a}} \left\lceil \frac{c}{\theta} \right\rceil - \left\lceil \frac{c}{a} \right\rceil  \leq 1$$

Moreover, we have that:

\begin{align*}
lim_{\theta \rightarrow 0} ~ h(\theta,\epsilon_1) =& +\infty \\
lim_{\theta \rightarrow +\infty} ~ h(\theta,\epsilon_1) =& \card(E_L)+\sum_{e \in E} I_e + 2 \sum_{e \in E} \left\lceil \frac{\lambda_e K}{2 \epsilon_1} \right\rceil
\end{align*} 

Therefore, because the function $\theta \mapsto h(\theta,\epsilon_1)$ is right continuous, for any $n \geq \card(E_L)+\sum_{e \in E} I_e + 2 \sum_{e \in E} \left\lceil \frac{\lambda_e K}{2 \epsilon_1} \right\rceil$\footnote{Remark that $\card(E_L)+\sum_{e \in E} I_e + 2 \sum_{e \in E} \left\lceil \frac{\lambda_e K}{2 \epsilon_1} \right\rceil \leq \Omega(\epsilon_1)$ so this constrain vanishes when we ask $n$ to be larger than $N(\epsilon_1)$} there exists a unique $\overline{\theta}$ which is the minimal solution to:

\begin{equation}\label{def:overxi}
n \leq h(\overline{\theta},\epsilon_1) \leq n+ \sum_{e \in E} I_e
\end{equation}

We now show that $\boldsymbol{x}(\overline{\theta},\epsilon_1)$ is an equilibrium with $n':=h(\overline{\theta},\epsilon_1)$ players. Suppose that $n \geq \Omega(\epsilon_1)$, where $\Omega(\epsilon_1)$ is defined in equation (\ref{def:M_esp_1}), then for all $e \in E$ and all $i \in \{0,\dots,I_e-1\}$, we have that:

$$h(\bar{\theta},\epsilon_1) \geq n \geq \Omega(\epsilon_1) \geq h(\frac{\ell_e g_e^i}{10},\epsilon_1) $$
where first inequality comes equation (\ref{def:overxi}) and last inequality is proved in Lemma \ref{le:comparison_N_h} in Annex and only requires computation. Because the function $\theta \rightarrow h(\theta,\epsilon_1)$ is decreasing it results that for all $e$ and $i$:

$$\overline{\theta} \leq \frac{\ell_e g_e^i}{10}$$

Therefore, Lemma \ref{le:xi_small_equilibrium} applies and proves that the profile $\boldsymbol{x}(\overline{\theta},\epsilon_1)$ is an equilibrium in the game $\mathcal{H}(n',S,g(\epsilon_1))$ with $n'$ players, where $n' \in [n, n+ \sum_{e \in E} I_e]$.\\

\textbf{\underline{Second step:}} We now provide a method to construct a profile $\hat{\boldsymbol{x}}(\epsilon_1)$ with exactly $n$ players and that is an equilibrium in  $\mathcal{H}(n,S,g(\epsilon_1))$. On the profile $\boldsymbol{x}(\overline{\theta},\epsilon_1)$ with $n'$ players, we have at most $\sum_e I_e$ too many players. This quantity is equal to the total number of steps on which $g(\epsilon_1)$ is constant. This is also the number of different steps in the construction of $\boldsymbol{x}(\overline{\theta},\epsilon_1)$. On each step, there are 2 players in the same location, at distance $\frac{2 \theta}{g_e^i}$ from their neighbors in both directions. We claim that one of these players is not necessary: it means that the same configuration without this player is still an equilibrium in the game $\mathcal{H}(n'-1,S,g(\epsilon_1))$. Indeed, in the new configuration, only the payoff of the player who shared his location with the removed player has changed. This payoff has increasing and is now equal to $2\theta$. A deviation to this location would give any player a payoff of $\theta$, which is not profitable. It results that there is no profitable deviation and the new profile of location is an equilibrium in the game $\mathcal{H}(n'-1,S,g(\epsilon_1))$. We can reiterate this argument with up to one player on each step, so in total with up to $\sum_e I_e$ players in the graph. We can therefore get an equilibrium in the game $\mathcal{H}(n,S,g(\epsilon_1))$ and this equilibrium is denoted $\hat{\boldsymbol{x}}(\epsilon_1)$.
\end{proof}

\subsection{$\epsilon$-equilibrium in the general game}\label{subsec:eps_eq}

In subsection \ref{subsec:eq_step} we constructed the profile $\hat{\boldsymbol{x}}(\epsilon_1)$ and proved that it is an (exact) equilibrium in $\mathcal{H}(n,S,g(\epsilon_1))$, provided that the number of  players $n$ satisfies $n \geq \Omega(\epsilon_1)$. In the current subsection, we prove that for any $\epsilon>0$ the profile $\hat{\boldsymbol{x}}(\epsilon_1)$ is moreover an $\epsilon$-equilibrium in $\mathcal{H}(n,S,f)$ in the case where $\epsilon_1$ is smaller than a given upper bound and where the number of  players $n$ continues to satisfy $n \geq \Omega(\epsilon_1)$.\\

Remember that the functions $p=(p_k)_{1 \leq k \leq n}$ and $\pi=(\pi_k)_{1 \leq k \leq n}$ are defined in equations (\ref{eq:payoff_function}) and  (\ref{eq:payoff_g}) as the payoff functions in the games $\mathcal{H}(n,S,f)$ and $\mathcal{H}(n,S,g(\epsilon_1))$, respectively. In the next lemma, the first claim expresses how close are the payoffs of the players when they play the profile $\boldsymbol{\hat{x}}(\epsilon_1)$ in $\mathcal{H}(n,S,g(\epsilon_1))$ or in $\mathcal{H}(n,S,f)$.\\

The second claim concerns a possible profile $x_{dev}$ with $n$ players after a deviation of a single player from $\hat{\boldsymbol{x}}(\epsilon_1)$. In other words if $\hat{\boldsymbol{x}}(\epsilon_1)=(x_1,\dots,x_n)$ then:

\begin{equation}\label{def:x_dev}
x_{dev}=(x_1,\dots,x_{k-1},y,x_{k+1},\dots,x_n)
\end{equation} 

for a given player $k$ and a given location $y \in G$.

\begin{lemma}\label{le:payoff_f_and_g_close}
For any $\epsilon_1 \in (0,\frac{\min \lambda_e K}{2})$, if $n \geq \Omega(\epsilon_1)$ then:

\begin{itemize}
\item Claim 1: $\boldsymbol{\hat{x}}(\epsilon_1)$ is such that for any player $k$ in $\{1,\dots,n\}$:

$$ \left| p_k(\boldsymbol{\hat{x}}(\epsilon_1)) - \pi_k(\boldsymbol{\hat{x}}(\epsilon_1)) \right| \leq 
\frac{4 \epsilon_1 \overline{\theta} }{m-\epsilon_1}$$

where $\overline{\theta}$ is defined in equation (\ref{def:overxi}).

\item Claim 2: the profile $x_{dev}$ defined in equation (\ref{def:x_dev}) is such that for any deviating player $k \in \{1,\dots,n\}$ we have:

$$ | p_k(x_{dev}) - \pi_k(x_{dev}) | \leq \frac{8 \epsilon_1 \overline{\theta} }{m-\epsilon_1}$$
\end{itemize}
\end{lemma}

\begin{proof} \underline{Proof of Claim 1:} It follows from the definitions of $p$ and $\pi$ that:

$$| p_k(\boldsymbol{\hat{x}}(\epsilon_1)) - \pi_k(\boldsymbol{\hat{x}}(\epsilon_1)) | = \left| \frac{1}{\card(\{j\in \{1,\dots,n \}: x_{j}=x_{k}\})}\sum_{A \subset \{x_1,\dots,x_n \} \atop x_{k}\in A } \frac{\int_{D_A} (g(\epsilon_1)(x)-f(x) ) d\mathcal{L}(x)}{\card(A)} \right|$$
where $D_A$ was defined in equation (\ref{def:C_A}). If we use the convention that $deg(x_k)=2$ if $x_k \in G \setminus V$ then the profile $\boldsymbol{\hat{x}}(\epsilon_1)$ is such that the number of players located in $x_k \in G$ is either $0$ or at least $deg(x_k)-1>0$. Using this property and the fact that $\| f - g(\epsilon_1)\|_{\infty} \leq \epsilon_1 $ (as proved in Lemma \ref{le:f-g}), we have that:

$$| p_k(\boldsymbol{\hat{x}}(\epsilon_1)) - \pi_k(\boldsymbol{\hat{x}}(\epsilon_1)) | \leq \frac{1}{
deg(x_k)-1}\sum_{A \subset \{x_1,\dots,x_n \} | x_{k}\in A } \frac{\mathcal{L}(D_A) \epsilon_1}{\card(A)} $$

The profile $\boldsymbol{\hat{x}}(\epsilon_1)$ also has the property that at least one player is located in every vertex. It implies that ties only occur between players in the same location.\footnote{We consider here only ties on a set with a strictly positive measure.} Therefore the sum $\displaystyle \sum_{A \subset \{x_1,\dots,x_n \} \atop x_k \in A}$ is reduced to a single term, where $A=\{ x_k \}$. It follows:

$$| p_k(\boldsymbol{\hat{x}}(\epsilon_1)) - \pi_k(\boldsymbol{\hat{x}}(\epsilon_1)) | \leq  \frac{1}{
deg(x_k)-1} \mathcal{L}(D_{\{x_k\}}) \epsilon_1 $$

It follows from the definition of $\boldsymbol{\hat{x}}(\epsilon_1)$ that the domain of attraction of the location $\{x_k \}$ is a union of $deg(x_k)$ intervals, and that such an interval has a length bounded by $\frac{2 \overline{\theta}}{g_e^i}$ for a given $e \in E$ and a given $i \in \{0,I_e-1\}$. It results:

$$| p_k(\boldsymbol{\hat{x}}(\epsilon_1)) - \pi_k(\boldsymbol{\hat{x}}(\epsilon_1)) | \leq  \frac{deg(x_k)}{
deg(x_k)-1}  \frac{2 \overline{\theta}}{g_e^i} \epsilon_1  $$

Because $deg(x_k)\geq 2$ and $g(\epsilon_1) \geq m-\epsilon_1$, we proved that:

$$| p_k(\boldsymbol{\hat{x}}(\epsilon_1)) - \pi_k(\boldsymbol{\hat{x}}(\epsilon_1)) | \leq   \frac{4 \epsilon_1 \overline{\theta}}{m-\epsilon_1}$$

\underline{Proof of claim 2:} Suppose first that the player $k$ deviates to a new location
$y$ that is not a leaf (i.e. a vertex with degree $1$). Using the same arguments detailed in the beginning of the proof of Claim 1 applied to $x_{dev}$, we obtain:

$$| p_k(x_{dev}) - \pi_k(x_{dev}) | \leq  \frac{1}{
deg(y)-1} \mathcal{L}(D_{\{x_k\}}) \epsilon_1 $$

It follows from the definition of $\boldsymbol{\hat{x}}(\epsilon_1)$ that, even after a unilateral deviation, the domain of attraction of the location $y$ is a union of $deg(y)$ intervals. Such an interval has a length at most twice as large as an interval in $\boldsymbol{\hat{x}}(\epsilon_1)$, i.e. a it has a length bounded above by $\frac{4 \overline{\theta}}{g_e^i}$ for a given $e \in E$ and a given $i \in \{0,I_e-1\}$. We get therefore:

$$| p_k(\boldsymbol{\hat{x}}(\epsilon_1)) - \pi_k(\boldsymbol{\hat{x}}(\epsilon_1)) | \leq  \frac{deg(x_k)}{
deg(x_k)-1}  \frac{4 \overline{\theta}}{g_e^i} \epsilon_1  $$

Because $deg(x_k)\geq 2$ and $g(\epsilon_1) \geq m-\epsilon_1$, we proved that:

$$| p_k(\boldsymbol{\hat{x}}(\epsilon_1)) - \pi_k(\boldsymbol{\hat{x}}(\epsilon_1)) | \leq   \frac{8 \epsilon_1 \overline{\theta}}{m-\epsilon_1}$$

If player $k$ deviates to a leaf $y$, then $\card(\{j\in \{1,\dots,n \}: x_{j}=y\})=1$ and its domain of attraction is a unique interval. We obtain  $|p_k(x_{dev}) - \pi_k(x_{dev}) | \leq \frac{4 \epsilon_1 \overline{\theta}}{m-\epsilon_1} \leq \frac{8 \epsilon_1 \overline{\theta}}{m-\epsilon_1}$ \end{proof}

\begin{proposition}\label{prop:psi_phi}
If $n \geq \Omega(\epsilon_1)$ then $\hat{\boldsymbol{x}}(\epsilon_1)$ is both an additive $\Psi[\epsilon_1]$-equilibrium, and a multiplicative $\Phi[\epsilon_1]$-equilibrium in $\mathcal{H}(n,S,f)$, i.e. for any $x_k \in G$:

$$ p_k(x_{dev}) \leq p_k(\hat{\boldsymbol{x}}(\epsilon_1))+\Psi(\epsilon_1) $$
and
$$ p_k(x_{dev}) \leq (1+\Phi(\epsilon_1)) p_k(\hat{\boldsymbol{x}}(\epsilon_1))$$
where $\Psi(\epsilon_1):=\frac{ 5\epsilon_1^2  (M+ \epsilon_1) }{12K(m-\epsilon_1)}$ and $\Phi(\epsilon_1):=\frac{12 \epsilon_1}{m-\epsilon_1}$.
\end{proposition}

\begin{proof} We have that:

$$ p_k(x_{dev})- p_k(\hat{\boldsymbol{x}}(\epsilon_1)) = \underbrace{p_k(x_{dev})- \pi_k(x_{dev})}_{A}+\underbrace{\pi_k(x_{dev})- \pi_k(\hat{\boldsymbol{x}}(\epsilon_1)) }_{B}+\underbrace{\pi_k(\hat{\boldsymbol{x}}(\epsilon_1)) -p_k(\hat{\boldsymbol{x}}(\epsilon_1)) }_{C}  $$

Because $\hat{\boldsymbol{x}}(\epsilon_1)$ is an equilibrium in $\mathcal{H}(n,S,g(\epsilon_1))$, we have that $B \leq 0$. Lemma \ref{le:payoff_f_and_g_close} proves that $C \leq \frac{4 \epsilon_1 \overline{\theta} }{m-\epsilon_1}$ and $A \leq \frac{8 \epsilon_1 \overline{\theta} }{m-\epsilon_1}$, we have therefore: 

\begin{equation}\label{eq:diff_payoff_avec_les_xi}
p_k(x_{dev})- p_k(\hat{\boldsymbol{x}}(\epsilon_1)) \leq \frac{12 \epsilon_1 \overline{\theta} }{m-\epsilon_1}
\end{equation}

We now use the fact that $n \geq \Omega(\epsilon_1)$ implies $\overline{\theta} \leq \frac{\epsilon_1(M+ \epsilon_1)}{5K}$, as proved by simple computation in Lemma \ref{le:ineq_xi} in Annex.

$$p_k(x_{dev})- p_k(\hat{\boldsymbol{x}}(\epsilon_1)) \leq \frac{ 5\epsilon_1^2  (M+ \epsilon_1) }{12K(m-\epsilon_1)}=\Psi(\epsilon_1)$$

On the other hand: $$\frac{p_k(x_{dev})}{p_k(\hat{\boldsymbol{x}}(\epsilon_1))}=     \frac{p_k(\hat{\boldsymbol{x}}(\epsilon_1)) + p_k(x_{dev})-p_k(\hat{\boldsymbol{x}}(\epsilon_1))}{p_k(\hat{\boldsymbol{x}}(\epsilon_1))}= 1+ \frac{p_k(x_{dev})-p_k(\hat{\boldsymbol{x}}(\epsilon_1))}{p_k(\hat{\boldsymbol{x}}(\epsilon_1))}$$

But, combining equation (\ref{eq:diff_payoff_avec_les_xi}) and Claim 3 of lemma \ref{le:4_claims}, we have that:

$$\frac{p_k(x_{dev})-p_k(\hat{\boldsymbol{x}}(\epsilon_1))}{p_k(\hat{\boldsymbol{x}}(\epsilon_1))} \leq \frac{12 \epsilon_1 \overline{\theta} }{(m-\epsilon_1)\overline{\theta}}=\frac{12 \epsilon_1}{m-\epsilon_1}=\Phi(\epsilon_1)$$
\end{proof}

Our main result (Theorem \ref{thm:main_thm}) is a straightforward consequence of Proposition \ref{prop:psi_phi}. Let indeed $\epsilon_1$ be such that $\Psi(\epsilon_1) \leq \epsilon$ (resp. $\Phi(\epsilon_1) \leq \epsilon$), it exists since $\Psi(\epsilon_1) \rightarrow 0$ and $\Phi(\epsilon_1) \rightarrow 0$ when $\epsilon_1 \rightarrow 0$, and let $n$ be such that $n \geq \Omega(\epsilon_1)$, then $\hat{\boldsymbol{x}}(\epsilon_1)$ is an additive (resp. multiplicative) $\epsilon$ equilibrium in $\mathcal{H}(n,S,f)$. \\

Because $\Phi(\epsilon_1) \leq \epsilon \Leftrightarrow \epsilon_1 \leq \frac{\epsilon m}{12+\epsilon}$, the existence of multiplicative equilibrium is guaranteed when the number of players is larger than $$N(\epsilon):=\Omega(\frac{\epsilon m}{12+\epsilon})=5 \card(E) +   \left\lceil  \frac{5L (M+\frac{\epsilon m}{12 + \epsilon})}{m-\frac{\epsilon m}{12 + \epsilon}}\left(\frac{(12 + \epsilon)K}{2 \epsilon m} + \frac{1}{\min \lambda_e}\right) +   \frac{3 L K (12 + \epsilon)}{2 \epsilon m} \right\rceil$$ where $L:= \int_S f(x)dx$ is the total quantity of consumers in the network, and $\min \lambda_e$ denotes the minimal length among the set $E$ of edges.

\section{Annex}\label{se:Annex}
\begin{lemma} \label{le:comparison_N_h}
Let $\epsilon_1 \in (0,\frac{\min \lambda_e K}{2})$, $\hat{e} \in E$ and $\hat{i} \in \{0,\dots,I_e-1\}$, then $\Omega(\epsilon_1) \geq h(\frac{\ell_{\hat{e}}g_{\hat{e}}^{\hat{i}}}{10},\epsilon_1)$ where $\Omega(\epsilon_1)$ is defined in equation (\ref{def:M_esp_1}).
\end{lemma}

\begin{proof}
It follows from the definition of $\ell_e$ and the inequality $m-\epsilon_1 \leq g_e \leq M + \epsilon_1$ that:

\begin{align*}
h(\frac{\ell_{\hat{e}}g_{\hat{e}}^{\hat{i}}}{8},\epsilon_1)=& 2 \card(E_L) + \sum_e \sum_i \left\lceil \frac{10}{2} \frac{\ell_e(\epsilon_1) g_e^i}{\ell_{\hat{e}(\epsilon_1)}g_{\hat{e}}^{\hat{i}}} \right\rceil + 2 \sum_e \left\lceil \frac{\lambda_e K}{2 \epsilon_1} \right\rceil \\
\leq & 2 \card(E_L) + \sum_e  \left\lceil \frac{\lambda_e K}{2 \epsilon_1} \right\rceil  \left\lceil 5 \frac{\lambda_e   \left\lceil\frac{\lambda_{\hat{e}}K}{2 \epsilon_1}\right\rceil(M+\epsilon_1)}{\lambda_{\hat{e}} \left\lceil\frac{\lambda_{e}K}{2 \epsilon_1}\right\rceil (m-\epsilon_1)} \right\rceil+ 2 \sum_e \left\lceil \frac{\lambda_e K}{2 \epsilon_1} \right\rceil
\end{align*}

Moreover, because $x \leq \left\lceil x \right\rceil \leq x+1$ we also have:
\begin{align*}
h(\frac{\ell_{\hat{e}}g_{\hat{e}}^{\hat{i}}}{8},\epsilon_1) \leq& 2 \card(E_L) + \sum_e   \left( 5 \frac{\lambda_e (\frac{\lambda_{\hat{e}}K}{2 \epsilon_1}+1) (M+\epsilon_1)}{\lambda_{\hat{e}} (m-\epsilon_1)} \right) +3 \sum_e \left\lceil \frac{\lambda_e K}{2 \epsilon_1} \right\rceil \\
\leq& 2 \card(E_L) + \frac{5}{2} \sum_e    \frac{\lambda_e K (M+\epsilon_1)}{\epsilon_1 (m-\epsilon_1)} +\sum_e 5 \frac{\lambda_e}{\lambda_{\hat{e}}}\frac{M+\epsilon_1}{m-\epsilon_1} +3 \sum_e \left\lceil \frac{\lambda_e K}{2 \epsilon_1} \right\rceil \\
\end{align*}

Finally, we use the notation $L:= \sum_e \lambda_e$ and the fact that  $\sum_e \left\lceil \frac{\lambda_e K}{2 \epsilon_1} \right\rceil \leq \card(E)+\frac{L K}{2 \epsilon_1}$.
\begin{align*}
h(\frac{\ell_{\hat{e}}g_{\hat{e}}^{\hat{i}}}{8},\epsilon_1) \leq &  2 \card(E_L) +   \frac{ 5 L K (M+\epsilon_1)}{2 \epsilon_1 (m-\epsilon_1)} +  \frac{5 L}{\min \lambda_e}\frac{M+\epsilon_1}{m-\epsilon_1} +3 \card (E) +   \frac{3 L K}{2 \epsilon_1}  \\
\leq& 5 \card(E) +   \frac{5 L (M+\epsilon_1)}{(m-\epsilon_1)}\left(\frac{K}{2 \epsilon_1} + \frac{1}{\min \lambda_e}\right) +   \frac{3 L K}{2 \epsilon_1} \leq \Omega(\epsilon_1)  \\
\end{align*}
\end{proof}

\begin{lemma}\label{le:ineq_xi}
If $n \geq \Omega(\epsilon_1)$ then $\overline{\theta} \leq \frac{\epsilon_1 (M+ \epsilon_1)}{4K}$
\end{lemma}

\begin{proof}
Remember that $\overline{\theta}$ was defined in \ref{def:overxi} as a real number such that $n \leq h(\overline{\theta}) \leq n+ \sum_{e \in E} I_e$. We suppose here that $n \geq \Omega(\epsilon_1)$ and we have, thanks to \ref{le:comparison_N_h}, $\Omega(\epsilon_1) \geq h(\frac{\ell_e g_e^i}{10},\epsilon_1)$. Putting all these inequalities together gives:

$$ h(\overline{\theta},\epsilon_1) \geq h(\frac{\ell_e g_e^i}{10},\epsilon_1)$$

Because $h$ is a decreasing function, and replacing $\ell$ by its definition in Definition \ref{def:ell_and_ie}, we get:

$$ \overline{\theta} \leq \frac{\lambda_e g_e^i}{10 \left\lceil \frac{\lambda_e K}{2 \epsilon_1}\right\rceil} \leq \frac{\epsilon_1 g_e^i}{5K}$$

Using the inequality $g_e^i \leq M+\epsilon$ we obtain:

$$ \overline{\theta} \leq \frac{\epsilon_1 (M+ \epsilon_1)}{5K}$$
\end{proof}

\bibliographystyle{plainnat}
\bibliography{nounif}

\end{document}